\documentclass{article}
\usepackage{amsmath,amsfonts,amssymb,amsthm,mathtools}
\usepackage{mathdots}
\usepackage{fullpage}
\usepackage{hyperref}
\usepackage{enumerate}
\usepackage[numbers]{natbib}

\usepackage{fixmath}
\usepackage{graphicx}
\usepackage{subfig}
\usepackage{float}

\usepackage{tikz}
\usetikzlibrary{graphs}
\usetikzlibrary{angles,quotes}
\usetikzlibrary{calc}
\usetikzlibrary{arrows}
\usetikzlibrary{arrows.meta}
\usetikzlibrary{decorations.markings}
\usetikzlibrary{fit}
\usetikzlibrary{positioning}
\usepackage{tikz-cd}

\newcommand{\Ker}{\mathrm{Ker}}

\newcommand{\GL}{\mathrm{GL}}
\newcommand{\PSL}{\mathrm{PSL}}
\newcommand{\PGL}{\mathrm{PGL}}

\newcommand{\F}{\mathbb{F}}
\newcommand{\R}{\mathbb{R}}

\newcommand{\comm}{\mathrm{Comm}}

\usepackage{amsbsy}
\newcommand\bs[1]{\boldsymbol{ #1}}

\newtheorem{definition}{Definition}[section]

\newtheorem{defn}[definition]{Definition}
\newtheorem{thm}[definition]{Theorem}
\newtheorem{prop}[definition]{Proposition}
\newtheorem{lem}[definition]{Lemma}
\newtheorem{cor}[definition]{Corollary}

\newtheorem{conj}[definition]{Conjectures}
\theoremstyle{definition}

\title{Reduced Power Graphs of $PGL_n(\F_q)$}
\author{Yilong Yang}
\date{\today}

\begin{document}
\maketitle


\abstract{
Given a group $G$, let us connect two non-identity elements by an edge if and only if one is a power of another. This gives a graph structure on $G$ minus identity, called the reduced power graph. It is conjectured by Akbari and Ashrafi that if a non-abelian finite simple group has connected reeuced power graph, then it must be an alternating group.

In this paper, we shall give a complete descprition about when the reduced power graphs of $PGL_n(\F_q)$ is connected for all $q$ and all $n\geq 3$. In particular, the conjectured by Akbari and Ashrafi is false. We shall also provide an upper bound in their diameters, and in case of disconnection, provide a description of all connected components.
}

\section{Introduction}

\subsection{Background and Main Results}

Graphical representations of an algebraic structure is always a fascinating research topic. The most well-known class of graphs is the Cayley graphs of finitely generated groups, which has been studied for a long time. Recently, some other graphical representations of groups have also been gathering more attention from researchers. Around 2000, the concept of a directed power graph was introduced by Kelarev and Quinn \cite{KQ00} as a new kinds of graphical representations, and later on Chakrabarty et al. \cite{CGS09} introduced the undirected power graph of a group.

\begin{defn}
Given a group $G$, the \emph{directed power graph} of $G$ has all elements of $G$ as vertices, and two distinct vertices $g,h$ are connected by an edge from $g$ to $h$ if and only if $h$ is a power of $g$. The \emph{(undirected) power graph} of $G$ is the same graph but we drop all directions on all edges. The \emph{reduced power graph} of $G$ is obtained from the undirected power graph by deleting the identity element.
\end{defn}

In this paper, by power graph we shall always mean the undirected power graph, unless specifically said otherwise.

Power graphs have been shown to capture many algebraic properties of a group. For example, the power graph of a finite group is a tree if and only if every element of the group is its own inverse \cite{TCS13}. If a finite group's power graph lack particular kinds of subgraphs, then its isomorphism types can be classified \cite{DEF14}. A finite group is a non-cyclic group of prime exponent if and only if its power graph is non-complete and minimally edge-connected \cite{Panda20}. For more results of this type, see \cite{KSCC}. 

One interesting aspect of the power graph is its connectivity. For finite groups, obviously power graphs are always connected,because every element is connected to the identity element. For this purpose, it is more interesting to consider the connectivity of the reduced power graph, i.e., power graph with the identity element removed. In particular, Akbari and Ashrafi proposed the following conjecture in 2015.

\begin{conj}[Akbari and Ashrafi,\cite{AA15}]
The reduced power graph of a non-abelian simple group $G$ is connected only if $G$ is isomorphic to some alternating group $A_n$.
\end{conj}

In this paper, we shall show that this conjecture is false.

\begin{thm}
\label{MainThm:n}
Let $\F_q$ be a field with characteristic $p$. For any $n\geq 3$ when $q\neq 2$, and for any $n\geq 6$ when $q=2$, the connected components of the reduced power graph of $\PGL_n(\F_q)$ are the following.
\begin{enumerate}
\item If $n$ is prime, and the minimal polynomial of $A$ is irreducible of degree $n$, then the images of all powers of $A$ that are not multiples of identity would form a connected component in the reduced power graph of $\PGL_n(\F_q)$.
\item If $n<p$, and $A$ has minimal polynomial $p(x)=(x-\lambda)^n$ for some $\lambda\in\F_q^*$, then the images of all powers of $A$ that are not multiples of identity would form a connected component in the reduced power graph of $\PGL_n(\F_q)$.
\item If $q-1$ is prime and $n<q$, and $A$ is diagonalizable with distinct eigenvalues, then the images of all powers of $A$ that are not multiples of identity would form a connected component in the reduced power graph of $\PGL_n(\F_q)$.
\item If $q=2$ and $n-1$ is prime, and the minimal polynomial of $A$ is $(x-1)p(x)$ where $p(x)$ is irreducible of degree $n-1$, then the images of all powers of $A$ that are not multiples of identity would form a connected component in the reduced power graph of $\PGL_n(\F_q)$.
\item If $q=2$, and $n$ or $n-1$ has a prime factor $p_1$ such that $p_0=2^{p_1}-1$ is also prime, and the characteristic polynomial of $A$ divides $p(x)=x^{p_0}-1$, then the images of all powers of $A$ that are not multiples of identity would form a connected component in the reduced power graph of $\PGL_n(\F_q)$. (Note that in this case $n$ or $n-1$ must be at most $p_0-1$.)
\item All other matrices are in the same connected component. If $q=2$ and $n\geq 6$, then this component has diameter at most $20$. If $q=3$ and $n=4$, then this component has diameter at most $20$. If $q\neq 2,n\geq 3$ and we are not in the previous case, this component has diameter at most $16$.
\end{enumerate}
\end{thm}

\begin{cor}
Let $\F_q$ be a field with odd characteristic $p$. Then $\PGL_n(\F_q)$ has connected reduced power graph if and only if $n$ is not a prime and $n>p$.

If $\F_q$ is a field with characteristic $2$ and $q\neq 2$, then $\PGL_n(\F_q)$ has connected reduced power graph if and only if $n$ is not a prime, $n>2$, and either $q-1$ is not prime or $n\geq q$.

Finally, $\PGL_n(\F_2)$ has disconnected reduced power graph if and only if at least one of the following three situations occurs.
\begin{enumerate}
\item $n$ or $n-1$ is prime.
\item $n$ has a prime factor $p_1$ such that $p_0=2^{p_1}-1$ is also a prime and $n\leq 2^{p_1}-2$.
\item $n-1$ has a prime factor $p_1$ such that $p_0=2^{p_1}-1$ is also a prime and $n\leq 2^{p_1}-1$.
\end{enumerate}
\end{cor}

Note that when $n$ and $q-1$ are coprime, then $\PGL_n(\F_q)$ is isomorphic to $\PSL_n(\F_q)$, which is always non-abelian simple when $n\geq 3$. Say when $q=3$ and $n=9$, then we have a non-abelian finite simple group $\PGL_9(\F_3)$ with connected reduced power graph. Similarly, $\PGL_{9}(\F_2),\PGL_{12}(\F_2)$ also have connected reduced power graphs. And these are not isomorphic to any alternating groups. These shows that the conjecture by Akbari and Ashrafi is false.

For the sake of completion, we also list the number of connected components and the diameter of each component for the reduced power graph of $\GL_3(\F_2),\GL_4(\F_2),\GL_5(\F_2)$. These can easily be verified using a computer or some adaptations of the methods in this paper.

\begin{enumerate}
\item The reduced power graph of $\GL_3(\F_2)$ is made of components with diameter $1$, each component is made of non-identity powers of $X\begin{bmatrix}1&1&\\&1&1\\&&1\end{bmatrix}X^{-1}$ for some invertible $X$.
\item The reduced power graph of $\GL_4(\F_2)$ has one component containing $\begin{bmatrix}1&1&&\\&1&&\\&&1&\\&&&1\end{bmatrix}$, and another component containing $\begin{bmatrix}1&1&&\\&1&&\\&&1&1\\&&&1\end{bmatrix}$, and the rest are isolated components as described by Proposition~\ref{prop:ConnObstIrred2}.
\item The reduced power graph of $\GL_5(\F_2)$ has one component containing $\begin{bmatrix}1&1&&&\\&1&&&\\&&1&&\\&&&1&\\&&&&1\end{bmatrix}$, and another component containing $\begin{bmatrix}1&1&&&\\&1&&&\\&&1&1&\\&&&1&\\&&&&1\end{bmatrix}$, and the rest are isolated components as described by Proposition~\ref{prop:ConnObstIrred}.
\end{enumerate}

Thus the connected components of the reduced power graph of $\PGL_n(\F_q)$ for all $n\geq 3$ has be described, and a diameter upper bound is established for each component.

\subsection{Notations and Outline of Proof}

We use $\F_q$ to denote the field with $q$ elements, and $\GL_n(\F_q)$ to denote the group of invertible $n\times n$ matrices on $\F_q$. We use $Z(\GL_n(\F_q))$ to denote the center of the group $\GL_n(\F_q)$, i.e., the subgroup made of scalar multiples of identity. We use $\PGL_n(\F_q)$ to denote the quotient group $\GL_n(\F_q)/Z(\GL_n(\F_q))$. Throughout this paper, we use $I_k$ to denote the $k\times k$ identity matrix.

In order to study the reduced power graphs of $\PGL_n(\F_q)$, we shall study a special graph on $\GL_n(\F_q)$, which has the reduced power graph of $\PGL_n(\F_q)$ as a quotient graph.

\begin{defn}
Given a group $G$, let $Z$ be its center. Then the projectively reduced power graph of $G$ is obtained from the undirected power graph of $G$ by deleting vertices in $Z$. 
\end{defn}

\begin{prop}
Given a group $G$, let $Z$ be its center. Let $\Gamma$ be projective reduced power graph of $G$, and let $\Gamma'$ be the reduced power graph of $G/Z$. If $\Gamma$ is connected, then $\Gamma'$ is connected with the same or less diameter.
\end{prop}
\begin{proof}
$\Gamma'$ is a quotient graph of $\Gamma$.
\end{proof}

Therefore, to establish connectivity, diameters and distances on the reduced power graph of $\PGL_n(\F_q)$, it is enough to establish the same statements in Theorem~\ref{MainThm:n} on the projectively reduced power graph of $\GL_n(\F_q)$.

In order to achieve connection, we first need to identify elements of $\GL_n(\F_q)$ that are connected to as much other elements as possible. Note that if $A$ connected to $B$ in the projectively reduced power graph of $\GL_n(\F_q)$, then $B$ must be a power or a root of $A$. Therefore $B$ must be in the centralizer of $A$. Consequently, to search for elements with many connections, we can try to find elements with large centralizers. We shall call these the \textit{pivot matrices}.

\begin{defn}
For a prime power $q$ and an integer $n\geq 2$, an $n\times n$ matrix $A$ over $\F_q$ is called a \emph{pivot matrix} if $A$ is diagonalizable over $\F_q$, and $A$ has exactly two eigenvalues not counting multiplicity. An $n\times n$ matrix $A$ over $\F_q$ is called a \emph{Jordan pivot matrix} if all eigenvalues of $A$ are $1$, and $\dim\Ker(A-I)=n-1$.
\end{defn}

In particular, pivot matrices are those similar to $\begin{bmatrix}x&&&&&\\&\ddots&&&&\\&&x&&&\\&&&y&&\\&&&&\ddots&\\&&&&&y\end{bmatrix}$ for some distinct $x,y\in\F_q^*$. Jordan pivot matrices are those similar to $\begin{bmatrix}1&&1\\&\ddots&\\&&1\end{bmatrix}$.

Our focus is on the pivot matrices when $q\neq 2$. But for $q=2$, there are no pivot matrices, so we have to resort to Jordan pivot matrices. (And sometimes, even for $q\neq 2$, connections via Jordan pivot matrices might also give a shorter distance.) All these matrices are important, because they all have big centralizers.

In particular, in Section~\ref{Sec:PivotAllConnect}, we shall show that all (Jordan) pivot matrices in the projectively reduced power graph of $\GL_n(\F_q)$ are connected to each other by some short path if $n$ is large enough. (We need $n\geq 3$ when $q\neq 2$, and $n\geq 4$ when $q=2$.)

In Section~\ref{sec:nConn}, we shall show that almost all matrices in the projectively reduced power graph of $\GL_n(\F_q)$ have paths to (Jordan) pivot matrices, with exceptions outlined in Section~\ref{sec:nDisconn}. These results will establish Theorem~\ref{MainThm:n}.

\begin{proof}[Proof of Theorem~\ref{MainThm:n}]
Combine Proposition~\ref{Prop:PivotDistanceAllCases} and Proposition~\ref{Prop:PathToPivotAllCases}.
\end{proof}

\section{Preliminary}

In this section, we compile many preliminary results about prime powers, finite fields and matrices over them. These are mostly elementary and are kept here for reference purpose.

\subsection{Prime powers}

\begin{lem}
\label{Lem:ConseqPP}
If $q,q-1$ are both prime powers, then either $q$ is Fermat prime, or $q-1$ is a Mersenne prime, or $q=9$.
\end{lem}
\begin{proof}
This is a easy and special case of Catalan's conjecture, proven by Mih\u{a}ilescu \cite{Mihai04}.
\end{proof}

\begin{lem}
\label{Lem:qnPrimeEvasionQ}
Suppose $q$ and $q-1$ are both prime powers, and $q\neq 2$. For any positive integer $n>1$, $q^n-1$ has a prime factor coprime to $q-1$, unless $q=3$ and $n=2$.
\end{lem}
\begin{proof}
Note that $q-1$ is a prime power. So if $q^n-1$ has no prime factor coprime to $q-1$, then $q^n$ is also a prime power, hence $q^n$ and $q^n-1$ are also consecutive prime powers.

Since $n>1$, $q^n$ cannot be prime. Since $q\neq 2$ and $n>1$, $q^n-1$ must have a proper factor $q-1$, so it also cannot be prime. So by Lemma~\ref{Lem:ConseqPP}, the only possibility is $q^n=9$.
%
%
%
%
\end{proof}

\subsection{Matrices over finite fields}

In this subsection, we focus on $n\times n$ matrices over $\F_q$ whose minimal polynomial is irreducible of degree $n$. Note that such a matrix must be similar to the campanion matrix of an irreducible polynomial.

\begin{lem}
\label{Lem:CompanionOrderBasic}
Suppose the characteristic polynomial of a matrix $C\in\GL_n(\F_q)$ is irreducible. Then polynomials of $C$ of degree at most $n-1$ form a field $\F_q[C]$ isomorphic to $\F_{q^n}$ under the usual matrix addition and matrix multiplication. In particular, $C$ is invertible, with multiplicative order dividing $q^n-1$.
\end{lem}
\begin{proof}
It is easy to very that this is a field with $q^n$ elements.
\end{proof}

\begin{lem}
\label{Lem:CompanionPower}
Suppose the characteristic polynomial of a matrix $C\in\GL_n(\F_q)$ is irreducible, and the multiplicative order of $C$ is $k$. For any positive integer $t$, the minimal polynomial of $C^t$ is irreducible. For any factor $t$ of $k$, if $\frac{k}{t}$ is a factor of $q^m-1$, then the minimal polynomial of $C^t$ is irreducible with degree at most $m$.
\end{lem}
\begin{proof}
Let the characteristic polynomial of $C$ be $f(x)$. Since this is irreducible, $f(x)$ is also the minimal polynomial for $C$.

Suppose the minimal polynomial of $C^t$ is reducible, say it is $g(x)h(x)$ for polynomials $g(x),h(x)$ of degree at least one. Then we must have $g(C^t)h(C^t)=0$, and thus the minimal polynomial $f(x)$ of $C$ must divide $g(x^t)h(x^t)$. Since $f(x)$ is irreducible, $f(x)$ must divide $g(x^t)$ or $h(x^t)$. But this means we have $g(C^t)=0$ or $h(C^t)=0$, contradicting the fact that $g(x)h(x)$ is supposed to be the minimal polynomial for $C^t$. Hence the minimal polynomial of $C^t$ is irreducible. 

Now suppose $t$ is a factor of $k$, then $C^t$ has order $\frac{k}{t}$, which is a factor of $q^m-1$, therefore in the finite field $\F_q[C]$, $C^t$ is contained in the field with $q^m-1$ elements. Hence its minimal polynomial has degree at most $m$.
\end{proof}

\begin{cor}
\label{Cor:CompanionPower}
Suppose the characteristic polynomial of a matrix $C\in\GL_n(\F_q)$ is irreducible, and the multiplicative order of $C$ is $k$. For any factor $t$ of $k$, $\frac{k}{t}$ is a factor of $q-1$ if and only if $C^t$ is a scalar multiple of identity.
\end{cor}

\begin{cor}
\label{Cor:CompanionSameOrderSameSize}
Suppose $C$ is an $n\times n$ matrix over $\F_q$, $n>1$, and its characteristic polynomial is irreducible. If $C^{p_1}$ is a scalar multiple of identity for a prime number $p_1$, then $n$ is the smallest positive integer such that $p_1(q-1)$ divides $q^n-1$.
\end{cor}
\begin{proof}

Let $n_1$ be the smallest positive integer such that $p_1(q-1)$ divides $q^{n_1}-1$. We shall show that $n=n_1$.

Since $n>1$ and the characteristic polynomial of $C$ is irreducible, by Corollary~\ref{Cor:CompanionPower}, the multiplicative order of $C$ cannot divide $q-1$, but must divide $q^n-1$. On the other hand, since $C^{p_1}$ is a scalar multiple of identity, $C^{p_1(q-1)}=I$. So the multiplicative order of $C$ must divide $p_1(q-1)$. Together, this means the least common multiple of $q-1$ and the multiplicative order of $C$ must be $p_1(q-1)$. Since the multiplicative order of $C$ and $q-1$ are factors of $q^n-1$, therefore their least common multiple $p_1(q-1)$ must divide $q^n-1$. So $n_1\leq n$.

Suppose $C$ has multiplicative order $k$. Note that since $C^{p_1}$ is a scalar multiple of identity, $C^{p_1(q-1)}=I$. So $k$ divides $p_1(q-1)$, and therefore divides $q^{n_1}-1$. By Lemma~\ref{Lem:CompanionPower}, this means the minimal polynomial of $C$ has degree at most $n_1$. But it has degree $n$. So $n\leq n_1$.
\end{proof}

\begin{lem}
\label{Lem:CompanionRoot}
Suppose the characteristic polynomial of a matrix $C\in\GL_n(\F_q)$ is irreducible, and the multiplicative order of $C$ is $k$. For any positive integer $t$, if $kt$ is a factor of $q^n-1$, then $C$ has a $t$-th root $M$ whose characteristic polynomial is also irreducible, and whose multiplicative order is $kt$.
\end{lem}
\begin{proof}
Since $\F_q[C]$ is a field with $q^n$ elements, its multiplicative group $\F_q[C]^*$ is cyclic. Therefore if $C$ has multiplicative order $k$, and $kt$ is a factor of $q^n-1$ for some positive integer $k$, then a $t$-th root $M$ whose multiplicative order is $kt$.

Now $\F_q[M]$ is a vector space over $\F_q$ with dimension at most $n$, but it contains $\F_q[C]$. Therefore $\F_q[M]=\F_q[C]$ and it is a field. Therefore the minimal polynomial of $M$ is irreducible and has degree $n$, which means it is also the characteristic polynomial of $M$.
\end{proof}

\begin{lem}
\label{Lem:PrimeCompanionHasGoodRoot}
Suppose $C$ is an $n\times n$ matrix over $\F_q$ whose characteristic polynomial is irreducible, and whose multiplicative order is a power of a prime $p_1$. Then for any prime $p_0\neq p_1$, $C$ has a $p_0$-th root. Furthermore, if $p_0$ divides $q^n-1$, then $C$ has a $p_0$-th root whose multiplicative order is a multiple of $p_0$.
\end{lem}
\begin{proof}
Note that $\F_q[C]$ is a field with $q^n$ elements. Suppose $p_0$ does not divide $q^n-1$. Then as the multiplicative group $\F_q[C]^*$ is cyclic with order $q^n-1$, every non-zero element in this field must have a unique $p_0$-th root.

Suppose $p_0$ divides $q^n-1$. Since $C$ has multiplicative order $p_1^k$ for some integer $k$, $p_1^k$ must also divide $q^n-1$. Since $p_0\neq p_1$, $p_0p_1^k$ must divide $q^n-1$. So by Lemma~\ref{Lem:CompanionRoot}, $C$ has a $p_0$-th root whose multiplicative order is a multiple of $p_0p_1^k$.
\end{proof}

\begin{lem}
\label{Lem:DoubleCompanionHasGoodRoot}
Suppose $C$ is an $n\times n$ matrix over $\F_q$ whose characteristic polynomial is irreducible, and whose multiplicative order is a power of a prime $p_1$. Then for any prime $p_0\neq p_1$ dividing $q^{2n}-1$, $\begin{bmatrix}C&\\&C\end{bmatrix}$ has a $p_0$-th root whose multiplicative order is a multiple of $p_0$.
\end{lem}
\begin{proof}
If $p_0$ divides $q^n-1$ as well, then by Lemma~\ref{Lem:PrimeCompanionHasGoodRoot}, $C$ has a $p_0$-th root $C'$ whose multiplicative order is a multiple of $p_0$. Then $\begin{bmatrix}C'&\\&C'\end{bmatrix}$ is the desired $p_0$-th root of $\begin{bmatrix}C&\\&C\end{bmatrix}$.

Now suppose $p_0$ does not divide $q^n-1$. Note that $\F_q[C]$ is a field with $q^n$ elements. Since $p_0\neq p_1$ is a factor of $q^{2n}-1$, we can find an irreducible factor of the polynomial $f(x)=x^{p_0}-C$ of degree $2$ over the field $\F_q[C]$. Let this polynomial be $g(x)=x^2+g_1(C)x+g_0(C)$ for some polynomials $g_1(x),g_0(x)$ over $\F_q$. Then $\begin{bmatrix}0&-g_0(C)\\I&-g_1(C)\end{bmatrix}$ is the desired $p_0$-th root of $\begin{bmatrix}C&\\&C\end{bmatrix}$.
\end{proof}

\subsection{Generalized Jordan canonical form over finite fields}

We have the following well-known results for matrices over any field.

\begin{thm}
For any matrix $A$ over a field $K$, $A$ is similar to a block diagonal matrix $\begin{bmatrix}J_1&&\\&\ddots&\\&&J_t\end{bmatrix}$ over $K$, such that each diagonal block is $J_i=\begin{bmatrix}C_i&N&&\\ &\ddots&\ddots&\\&&\ddots&N\\&&&C_i\end{bmatrix}$ for a companion matrix $C_i$ of some irreducible polynomial over $K$, and $N=\begin{bmatrix}0&0&\dots&0\\\vdots&\vdots&\iddots&\vdots\\ 0&0&\dots&0\\1&0&\dots&0\end{bmatrix}$.
\end{thm}

This block diagonal matrix is called the generalized Jordan canonical form of $A$ over the field $K$, and it is unique up to permutations of the diagonal blocks. For more detail, see for example \cite{Perlis52}.

Since the focus of this paper is on matrices over finite fields, we can actually obtain a better result, by combining the following two statements.

\begin{prop}
Every irreducible polynomial over a finite field is separable.
\end{prop}
\begin{proof}
Commonly found in abstract algebra textbooks, say \cite{DF91}.
\end{proof}

\begin{prop}[\cite{Robinson70}]
Let $C$ be a companion matrix of an irreducible polynomial $f(x)$ over a field $K$. Again let $N=\begin{bmatrix}0&0&\dots&0\\\vdots&\vdots&\dots&\vdots\\ 0&0&\dots&0\\1&0&\dots&0\end{bmatrix}$. Then $\begin{bmatrix}C&N&&\\ &\ddots&\ddots&\\&&\ddots&N\\&&&C\end{bmatrix}$ and $\begin{bmatrix}C&I&&\\ &\ddots&\ddots&\\&&\ddots&I\\&&&C\end{bmatrix}$ are similar if and only if $f(x)$ is separable.
\end{prop}

Combining these two statements, we have the follower prettier version of generalized Jordan canonical form over finite fields.

\begin{thm}
For any matrix $A$ over a finite field $K$, $A$ is similar to a block diagonal matrix $\begin{bmatrix}J_1&&\\&\ddots&\\&&J_t\end{bmatrix}$ over $K$, such that each diagonal block is $J_i=\begin{bmatrix}C_i&I&&\\ &\ddots&\ddots&\\&&\ddots&I\\&&&C_i\end{bmatrix}$ for a companion matrix $C_i$ of some irreducible polynomial over $K$.
\end{thm}

This version of canonical form is much easier to calculate.
We establish some easy results here for later use.

\begin{lem}
\label{Lem:CompanionDecomp}
If $A$ is an $n\times n$ matrix whose minimal polynomial is an irreducible polynomial over $\F_q$ of degree $n$, and $B = A^k$ for some positive integer $k$. Then the generalized Jordan canonical form of $B$ is block diagonal with identical diagonal blocks, each block is a companion matrix to the same irreducible polynomial.
\end{lem}
\begin{proof}
This is simply Lemma~\ref{Lem:CompanionPower} in terms of generalized Jordan canonical form.
%
%
\end{proof}

\begin{lem}
\label{Lem:IrreducibleRootIrreducible}
If $A$ is an $n\times n$ matrix whose characteristic polynomial is an irreducible or has an irreducible factor of degree $n-1$, and $B^k = A$ for some positive integer $k$. Then the characteristic polynomial of $B$ is irreducible or has an irreducible factor of degree $n-1$, respectively.
\end{lem}
\begin{proof}
We prove the contrapositives. If $B$ has reducible characteristic polynomial, then the generalized Jordan canonical form of $B$ will be $\begin{bmatrix}C_1&*&*\\&\ddots&*\\&&C_t\end{bmatrix}$ for some companion matrices $C_1,\dots,C_t$ and $t\geq 2$. Therefore $A$ will be similar to $\begin{bmatrix}C_1^k&*&*\\&\ddots&*\\&&C_t^k\end{bmatrix}$, and thus its characteristic polynomial is also reducible.

The other statement about $n-1$ can be proven similarly.
\end{proof}

\section{Connections between Pivot Matrices}
\label{Sec:PivotAllConnect}

In this section, we shall show that all (Jordan) pivot matrices are connected by a short path in the projectively reduced power graph of $\GL_n(\F_q)$, when $n$ is not too small. In particular, we should establish the following.

\begin{prop}
\label{Prop:PivotDistanceAllCases}
For $q\neq 2$ and $n\geq 3$, in the projectively reduced power graph of $\GL_n(\F_q)$, all pivot matrices has distance at most $8$ to each other. 

For $q\neq 2$ and $n\geq 3$ or for $q=2$ and $n\geq 4$, in the projectively reduced power graph of $\GL_n(\F_q)$, all Jordan pivot matrices has distance at most $8$ to each other.
\end{prop}
\begin{proof}
Combine Proposition~\ref{Prop:PivotJordan}, Proposition~\ref{Prop:Pivot>=4} and Proposition~\ref{Prop:Pivot3}.
\end{proof}

As a special remark, the proofs here rely fundamentally on the fact that these (Jordan) piviot matrices have big centralizers. In fact, every matrix in the group $\GL_n(\F_q)$ can be written as the product of at most five matrices in the centralizers of some Jordan pivot matrices and some pivot matrices. 

\subsection{From Jordan to Jordan}

We only consider the cases for $q\neq 2$ and $n\geq 3$, or for $q=2$ and $n\geq 4$.

We consider the group $\GL_n(\F_q)$. In this group, pick $J=\begin{bmatrix}1&&1\\&\ddots&\\&&1\end{bmatrix}$, which is a Jordan pivot matrix. We also pick $A=\begin{bmatrix}1&&&&\\&x&&&\\&&1&&\\&&&\ddots&\\&&&&1\end{bmatrix}$, where $x\neq 1\in\F_q^*$ for $q\neq 2$, and $x$ is a block $\begin{bmatrix}1&1\\1&0\end{bmatrix}$ when $q=2$. Note that for $q\neq 2$ and $n\geq 3$, or for $q=2$ and $n\geq 4$, we always have $AJ=JA$. We also always have $A^{q^2-1}=I$.

\begin{lem}
For $q\neq 2$ and $n\geq 3$ or for $q=2$ and $n\geq 4$, $A,J$ are both powers of $AJ$. In particular, $A$ and $J$ have distance $2$ in the projectively reduced power graph of $\GL_n(\F_q)$.
\end{lem}
\begin{proof}
Let $p$ be the characteristic of the field $\F_q$. Since $AJ=JA$, $J^p=I$ and $A^{q^2-1}=I$, we have $(AJ)^{q^2}=A$.

On the other hand, $J^{q^2-1}=J^{-1}$. So we have $(AJ)^{(q^2-1)^2}=J$.
\end{proof}

Let $\comm(A)$ and $\comm(J)$ be the centralizer of the two elements in $\GL_n(\F_q)$ respectively. We adopt the notation that for any two subsets $H,K$ of a group $G$, then $HK$ stands for the the subset $\{hk\in G:h]in H,k\in K\}$. Our goal here is to show that these two centralizers can generate $\GL_n(\F_q)$ quickly, and specifically $\GL_n(\F_q)=\comm(J)\comm(A)\comm(J)\comm(A)\comm(J)$. Our first step is to show that, most of the time, $\comm(J)\comm(A)\comm(J)$ is enough.

\begin{lem}
\label{Lem:PivotJAJ=Strong}
For $q\neq 2$ and $n\geq 3$ or for $q=2$ and $n\geq 4$, suppose a matrix $X\in \GL_n(\F_q)$ has non-zero lower left entry and invertible lower left $(n-1)\times(n-1)$ block, then $X\in\comm(J)\comm(A)\comm(J)$.
\end{lem}
\begin{proof}
This is essentially the LU decomposition.

Pick $T=\begin{bmatrix}&&1\\&I_{n-2}&\\1&&\end{bmatrix}\in\comm(A)$. Then $TX$ has non-zero upper left entry, and invertible upper left $(n-1)\times(n-1)$ block. Consequently, it has a unique block LDU decomposition $TX=LDU$ where $L=\begin{bmatrix}1&&\\ *&I_{n-2}&\\ *& *&1\end{bmatrix}$ and $U=\begin{bmatrix}1& *& * \\ &U_{n-2}& * \\ &&1\end{bmatrix}$ and $D=\begin{bmatrix}d_1&&\\&I_{n-2}&\\&&d_n\end{bmatrix}$ for some $(n-2)\times(n-2)$ invertible matrix $U_{n-2}$ and some $d_1,d_n\in\F_q^*$. Note that $TLT,U\in\comm(J)$ and $TD\in\comm(A)$.

So $X=TLDU=(TLT)(TD)U\in\comm(J)\comm(A)\comm(J)$.
\end{proof}

Our second step is to show that every matrix can be adjusted to the case as described in Lemma~\ref{Lem:PivotJAJ=Strong}.

\begin{lem}
\label{Lem:PivotStrongJA=All}
For $q\neq 2$ and $n\geq 3$ or for $q=2$ and $n\geq 4$, and any $X\in\GL_n(\F_q)$, we can find a unit upper triangular matrix $U$ and a matrix $T=\begin{bmatrix}&&1\\&I_{n-2}&\\1&&\end{bmatrix}$, such that $XUT$ has non-zero lower left entry, and invertible lower left $(n-1)\times(n-1)$ block.
\end{lem}
\begin{proof}
Consider $X=\begin{bmatrix}*&*&*\\ \bs x_1&\dots&\bs x_n\end{bmatrix}$ where $\bs x_1,\dots,\bs x_n\in\F_q^{n-1}$. Since $X$ is invertible, these vectors cannot all have zero last coordinate. So $\bs y=\bs x_n+u_1\bs x_1+\dots+u_{n-1}\bs x_{n-1}$ will have non-zero last coordinate for some constants $u_1,\dots,u_{n-1}\in\F_q$. Set $U_1=\begin{bmatrix}1&&&u_1\\&\ddots&&\vdots\\&&1&u_{n-1}\\&&&1\end{bmatrix}$. Then $XU_1=\begin{bmatrix}*&*&*&*\\ \bs x_1&\dots&\bs x_{n-1}&\bs y\end{bmatrix}$ will have non-zero lower right entry.

Let $V$ be the space spanned by $\bs x_2,\dots,\bs x_{n-1},\bs y$. If $V=\F_q^{n-1}$, then $XU_1$ will also have invertible lower right $(n-1)\times(n-1)$ block. So $XU_1T$ is as desired.

Suppose that is not the case. Since $X$ is invertible, we must have $\dim V= n-2$, and $\bs x_1\notin V$. Now $\bs y$ has non-zero last coordinate, so $\bs y\neq \bs 0$. Since the collection $\bs x_2,\dots,\bs x_{n-1},\bs y$ spans $V$ and since $\bs y\neq \bs 0$, therefore we can find a subcollection of $n-2$ vectors including $\bs y$ that already spans $V$. Let $\bs x_i$ be the excluded vector for some $2\leq i\leq n-1$. 

Then since $\bs x_1\notin V$, we must have $\bs x_i+\bs x_1\notin V$. So if we set $U_2$ be the unit upper triangular matrix that corresponds to the elementary column operation of adding the first column to the $i$-th column, then $XU_1U_2$ will now have non-zero lower right entry and invertible lower right $(n-1)\times(n-1)$ block. So $XU_1U_2T$ is as desired.
\end{proof}

\begin{cor}
\label{Lem:PivotJAJAJ}
For $q\neq 2$ and $n\geq 3$ or for $q=2$ and $n\geq 4$, \[\GL_n(\F_q)=\comm(J)\comm(A)\comm(J)\comm(A)\comm(J).\]
\end{cor}
\begin{proof}
Pick any $X\in\GL_n(\F_q)$. By Lemma~\ref{Lem:PivotStrongJA=All}, we can find a unit upper triangular matrix $U$ and $T=\begin{bmatrix}&&1\\&I_{n-2}&\\1&&\end{bmatrix}$, such that $XUT$ has non-zero lower left entry, and invertible lower left $(n-1)\times(n-1)$ block. Then by Lemma~\ref{Lem:PivotJAJ=Strong}, we see that $XUT\in\comm(J)\comm(A)\comm(J)$. However, we also see that $U\in\comm(J)$ and $T\in\comm(A)$. So our result follows.
\end{proof}

\begin{prop}
\label{Prop:PivotJordan}
For $q\neq 2$ and $n\geq 3$ or for $q=2$ and $n\geq 4$, any two Jordan pivot matrices in the projectively reduced power graph of $\GL_n(\F_q)$ will have distance at most $8$.
\end{prop}
\begin{proof}
Suppose $J_1,J_2$ are two Jordan pivot matrices. Say $J_1=X_1JX_1^{-1}$ and $J_2=X_2JX_2^{-1}$ for invertible matrices $X_1,X_2$. Set $X=X_1^{-1}X_2$.

By Lemma~\ref{Lem:PivotJAJAJ}, we have $X=M_1M_2M_3M_4M_5$ where $M_1,M_3,M_5\in\comm(J)$ and $M_2,M_4\in\comm(A)$.

Now we have the following path.
\begin{align*}
J_1=&X_1JX_1^{-1}=X_1M_1JM_1^{-1}X_1^{-1}\\
\xrightarrow{\text{distance 2}}&X_1M_1AM_1^{-1}X_1^{-1}=X_1M_1M_2AM_2^{-1}M_1^{-1}X_1^{-1}\\
\xrightarrow{\text{distance 2}}&X_1M_1M_2JM_2^{-1}M_1^{-1}X_1^{-1}=X_1M_1M_2M_3JM_3^{-1}M_2^{-1}M_1^{-1}X_1^{-1}\\
\xrightarrow{\text{distance 2}}&X_1M_1M_2M_3AM_3^{-1}M_2^{-1}M_1^{-1}X_1^{-1}=X_1M_1M_2M_3M_4AM_4^{-1}M_3^{-1}M_2^{-1}M_1^{-1}X_1^{-1}\\
\xrightarrow{\text{distance 2}}&X_1M_1M_2M_3M_4JM_4^{-1}M_3^{-1}M_2^{-1}M_1^{-1}X_1^{-1}=X_1M_1M_2M_3M_4M_5JM_5^{-1}M_4^{-1}M_3^{-1}M_2^{-1}M_1^{-1}X_1^{-1}\\
=&X_1XJX^{-1}X_1^{-1}=X_2JX_2^{-1}=J_2.
\end{align*}
\end{proof}

\subsection{From Pivot to Pivot when $n\geq 4$}

When $q\neq 2$, it is more important to have a path between pivot matrices. For now, we only consider the cases for $q\neq 2$ and $n\geq 4$.

As before, we consider the group $\GL_n(\F_q)$. In this group, pick $J=\begin{bmatrix}1&&1\\&\ddots&\\&&1\end{bmatrix}$, which is a Jordan pivot matrix. We also pick $A=\begin{bmatrix}1&&&&\\&x&&&\\&&1&&\\&&&\ddots&\\&&&&1\end{bmatrix}$, where $x\neq 1\in\F_q^*$. 

We first establish a version of Lemma~\ref{Lem:PivotJAJ=Strong} with a weaker requirement. However, this only works for $n\geq 4$.

\begin{lem}
\label{Lem:PivotJAJ=Weak}
For $q\neq 2$ and $n\geq 4$, suppose a matrix $X\in \GL_n(\F_q)$ has non-zero lower left entry, then $X\in\comm(J)\comm(A)\comm(J)$.
\end{lem}
\begin{proof}
Since the lower left entry of $X$ is non-zero, using elementary row and column operations, we can find unit upper triangular matrices $U_1,U_2$ such that $U_1XU_2=\begin{bmatrix}0&\begin{matrix}*&\dots\end{matrix}&*\\ \begin{matrix}\vdots\\0\end{matrix}&X'&\begin{matrix}\vdots\\ *\end{matrix}\\x_{n1}&\begin{matrix}0&\dots\end{matrix}&0\end{bmatrix}$ for some $(n-2)\times(n-2)$ matrix $X'$ and some $x_{n1}\neq 0$. 

Since $U_1XU_2$ is invertible, its upper right $(n-1)\times(n-1)$ block will have rank $n-1$. Hence $X'$ has rank at least $n-3$. When $n\geq 4$, this implies that $X'$ is not zero. So $R'_1X'R'_2=\begin{bmatrix}1&\\&X''\end{bmatrix}$ for some invertible $(n-2)\times(n-2)$ matrices $R'_1,R'_2$ and some $(n-3)\times(n-3)$ matrix $X''$. Set $R_1=\begin{bmatrix}1&&\\&R'_1&\\&&1\end{bmatrix}$ and $R_2=\begin{bmatrix}1&&\\&R'_2&\\&&1\end{bmatrix}$, then $R_1U_1XU_2R_2=\begin{bmatrix}
0& *&\begin{matrix} *&\dots& *\end{matrix}& *\\0&1&\begin{matrix} 0&\dots& 0\end{matrix}& *\\ \begin{matrix}0\\ \vdots\\ 0\end{matrix}& \begin{matrix} 0\\ \vdots\\ 0\end{matrix} & X''&\begin{matrix} *\\ \vdots\\ *\end{matrix}\\x_{n1}&0&\begin{matrix} 0&\dots& 0\end{matrix}&0\end{bmatrix}$.

Then using elementary row and column operations, we can find unit upper triangular matrices $U_3,U_4$ such that $U_3R_1U_1XU_2R_2U_4=\begin{bmatrix}
0& 0&\begin{matrix} *&\dots& *\end{matrix}& *\\0&1&\begin{matrix} 0&\dots& 0\end{matrix}& 0\\ \begin{matrix}0\\ \vdots\\ 0\end{matrix}& \begin{matrix} 0\\ \vdots\\ 0\end{matrix} & X''&\begin{matrix} *\\ \vdots\\ *\end{matrix}\\x_{n1}&0&\begin{matrix} 0&\dots& 0\end{matrix}&0\end{bmatrix}\in\comm(A)$.

Since $U_1,U_2,U_3,U_4,R_1,R_2\in\comm(J)$, we are done.
\end{proof}


Now a pivot matrix is diagonalizable, but its diagonal form does not necessarily looks like $A$. So the analogue to Lemma~\ref{Lem:PivotJAJAJ} are the followings.

\begin{lem}
\label{Lem:PivotAJAJAUnbalanced}
For $q\neq 2$ and $n\geq 3$, in the group $\GL_n(\F_q)$, we pick an arbitrary $A_1=\begin{bmatrix}x_1I_{n-m_1-1}&&\\&y_1I_{m_1}&\\&&x_1\end{bmatrix}$ and $A_2=\begin{bmatrix}x_2I_{n-m_2-1}&&\\&y_2I_{m_2}&\\&&x_2\end{bmatrix}$ for some positive integers $m_1,m_2\leq n-2$, and some $x_1,x_2,y_1,y_2\in\F_q^*$ such that $x_1\neq y_1$ and $x_2\neq y_2$. 

Suppose $m_1+m_2<n$, then $\GL_n(\F_q)=\comm(A_1)\comm(J)\comm(A)\comm(J)\comm(A_2)$.

Suppose $m_1=\frac{n}{2}$, then $\GL_n(\F_q)$ is the union of $\comm(A_1)\comm(J)\comm(A)\comm(J)\comm(A_2)$ and $P\comm(A_1)\comm(J)\comm(A)\comm(J)\comm(A_2)$ for a permutation matrix $P$ such that $PA_1P^{-1}=\begin{bmatrix}y_1I_{n-m_1-1}&&\\&x_1I_{m_1}&\\&&y_1\end{bmatrix}$.
\end{lem}
\begin{proof}
For any $X\in\GL_n(\F_q)$, we break it down into block form $X=\begin{bmatrix}X_{11}&X_{12}&X_{13}\\X_{21}&X_{22}&X_{23}\\X_{31}&X_{32}&X_{33}\end{bmatrix}$ where the rows are grouped into the first $n-m_1-1$ rows, the next $m_1$ rows, and the last row, and the columns are grouped into the first $n-m_2-1$ columns, the next $m_2$ columns, and the last column.

The submatrix $\begin{bmatrix}X_{11}&X_{13}\\X_{31}&X_{33}\end{bmatrix}$ has rank at least $n-m_1-m_2$. If $m_1+m_2<n$, then this submatrix is non-zero. So we can find invertible matrices $\begin{bmatrix}R_{11}&R_{12}\\R_{21}&R_{22}\end{bmatrix}$ and $\begin{bmatrix}R'_{11}&R'_{12}\\R'_{21}&R'_{22}\end{bmatrix}$ such that $\begin{bmatrix}R_{11}&R_{12}\\R_{21}&R_{22}\end{bmatrix}\begin{bmatrix}X_{11}&X_{13}\\X_{31}&X_{33}\end{bmatrix}\begin{bmatrix}R'_{11}&R'_{12}\\R'_{21}&R'_{22}\end{bmatrix}$ has a non-zero lower left entry. 

Let $R=\begin{bmatrix}R'_{11}&&R'_{12}\\&I_{m_1}&\\R'_{21}&&R'_{22}\end{bmatrix}$ and $R'=\begin{bmatrix}R'_{11}&&R'_{12}\\&I_{m_2}&\\R'_{21}&&R'_{22}\end{bmatrix}$. Then $RXR'$ will have non-zero lower left entry. So by Lemma~\ref{Lem:PivotJAJ=Weak}, $RXR'\in\comm(J)\comm(A)\comm(J)$. Since $R\in\comm(A_1)$ and $R'\in\comm(A_2)$, we are done.

Now suppose $m_1=\frac{n}{2}$. Then $\begin{bmatrix}X_{13}\\X_{23}\\X_{33}\end{bmatrix}$ is in fact the last column of $X$. Since $X$ is invertible, entries in this column cannot all be zero. Let $P$ be any permutation matrix such that $PA_1P^{-1}=\begin{bmatrix}y_1I_{n-m_1-1}&&\\&x_1I_{m_1}&\\&&y_1\end{bmatrix}$. Then $P$ swaps the $\frac{n}{2}$ rows involved in $X_{23}$ with the $\frac{n}{2}$ rows involved in $X_{13}$ and $X_{33}$. Then $X$ or $PX$ will have a block structure $\begin{bmatrix}X'_{11}&X'_{12}&X'_{13}\\X'_{21}&X'_{22}&X'_{23}\\X'_{31}&X'_{32}&X'_{33}\end{bmatrix}$ where the submatrix $\begin{bmatrix}X'_{11}&X'_{13}\\X'_{31}&X'_{33}\end{bmatrix}$ is non-zero. 

So we can find invertible matrices $\begin{bmatrix}R_{11}&R_{12}\\R_{21}&R_{22}\end{bmatrix}$ and $\begin{bmatrix}R'_{11}&R'_{12}\\R'_{21}&R'_{22}\end{bmatrix}$ such that $\begin{bmatrix}R_{11}&R_{12}\\R_{21}&R_{22}\end{bmatrix}\begin{bmatrix}X'_{11}&X'_{13}\\X'_{31}&X'_{33}\end{bmatrix}\begin{bmatrix}R'_{11}&R'_{12}\\R'_{21}&R'_{22}\end{bmatrix}$ has a non-zero lower left entry. 

Let $R=\begin{bmatrix}R'_{11}&&R'_{12}\\&I_{n/2}&\\R'_{21}&&R'_{22}\end{bmatrix}$ and $R'=\begin{bmatrix}R'_{11}&&R'_{12}\\&I_{n/2}&\\R'_{21}&&R'_{22}\end{bmatrix}$. Then $RXR'$ or $RPXR'$ will have non-zero lower left entry. So by Lemma~\ref{Lem:PivotJAJ=Weak}, $RXR'\in\comm(J)\comm(A)\comm(J)$ or $RPXR'\in\comm(J)\comm(A)\comm(J)$. 

Since $R\in\comm(A_1)$ and $R'\in\comm(A_2)$, we see that $X$ or $PX$ is in $\comm(A_1)\comm(J)\comm(A)\comm(J)\comm(A_2)$. So we are done.
\end{proof}

\begin{lem}
Suppose $q\neq 2$ and $n\geq 3$. Let $A'$ be a diagonal matrix whose upper left entry and lower right entry are identical. Then $A',J$ are both powers of $A'J$. In particular, $A'$ and $J$ have distance $2$ in the projectively reduced power graph of $\GL_n(\F_q)$.
\end{lem}
\begin{proof}
Let $p$ be the characteristic of the field $\F_q$. Since $A_iJ=JA_i$, $J^p=I$ and $A_i^{q-1}=I$, we have $(A_iJ)^{q}=A_i$.

On the other hand, $J^{q-1}=J^{-1}$. So we have $(A_iJ)^{(q-1)^2}=J$.
\end{proof}

%

\begin{prop}
\label{Prop:Pivot>=4}
For $q\neq 2$ and $n\geq 4$, any two pivot matrices in the projectively reduced power graph of $\GL_n(\F_q)$ will have distance at most $8$.
\end{prop}
\begin{proof}
Suppose we have two Jordan pivot matrices. Say they are $X_1A_1X_1^{-1}$ and $X_2A_2X_2^{-1}$ for invertible matrices $X_1,X_2$, and $A_1=\begin{bmatrix}x_1I_{n-m_1-1}&&\\&y_1I_{m_1}&\\&&x_1\end{bmatrix}$ and $A_2=\begin{bmatrix}x_2I_{n-m_2-1}&&\\&y_2I_{m_2}&\\&&x_2\end{bmatrix}$ for some integers $m_1,m_2\leq n-2$, and some $x_1,x_2,y_1,y_2\in\F_q^*$ such that $x_1\neq y_1$ and $x_2\neq y_2$. By choosing $X_1,X_2$ carefully, we may assume that $m_1,m_2\leq\frac{n}{2}$. Then we must either have $m_1+m_2<n$, or we must have $m_1=m_2=\frac{n}{2}$.

Set $X=X_1^{-1}X_2$. By Lemma~\ref{Lem:PivotAJAJAUnbalanced}, we must either have $X=M_1M_2M_3M_4M_5$ or have $PX=M_1M_2M_3M_4M_5$, where $M_1\in\comm(A_1),M_3\in\comm(A),M_5\in\comm(A_2)$ and $M_2,M_4\in\comm(J)$, and $P$ is a permutation matrix such that $PA_1P^{-1}=\begin{bmatrix}y_1I_{n-m_1-1}&&\\&x_1I_{m_1}&\\&&y_1\end{bmatrix}$.

Suppose we have $X=M_1M_2M_3M_4M_5$. Then we have the following path.
\begin{align*}
X_1A_1X_1^{-1}=&X_1M_1A_1M_1^{-1}X_1^{-1}\\
\xrightarrow{\text{distance 2}}&X_1M_1JM_1^{-1}X_1^{-1}=X_1M_1M_2JM_2^{-1}M_1^{-1}X_1^{-1}\\
\xrightarrow{\text{distance 2}}&X_1M_1M_2AM_2^{-1}M_1^{-1}X_1^{-1}=X_1M_1M_2M_3AM_3^{-1}M_2^{-1}M_1^{-1}X_1^{-1}\\
\xrightarrow{\text{distance 2}}&X_1M_1M_2M_3JM_3^{-1}M_2^{-1}M_1^{-1}X_1^{-1}=X_1M_1M_2M_3M_4JM_4^{-1}M_3^{-1}M_2^{-1}M_1^{-1}X_1^{-1}\\
\xrightarrow{\text{distance 2}}&X_1M_1M_2M_3M_4A_2M_4^{-1}M_3^{-1}M_2^{-1}M_1^{-1}X_1^{-1}=X_1M_1M_2M_3M_4M_5A_2M_5^{-1}M_4^{-1}M_3^{-1}M_2^{-1}M_1^{-1}X_1^{-1}\\
=&X_1XA_2X^{-1}X_1^{-1}=X_2A_2X_2^{-1}.
\end{align*}

Now suppose we have $PX=M_1M_2M_3M_4M_5$. Let $A'_1=PA_1P^{-1}$. Then we have the following path.

\begin{align*}
X_1A_1X_1^{-1}=&X_1P^{-1}A'_1PX_1^{-1}=X_1P^{-1}M_1A'_1M_1^{-1}PX_1^{-1}\\
\xrightarrow{\text{distance 2}}&X_1P^{-1}M_1JM_1^{-1}PX_1^{-1}=X_1P^{-1}M_1M_2JM_2^{-1}M_1^{-1}PX_1^{-1}\\
\xrightarrow{\text{distance 2}}&X_1P^{-1}M_1M_2AM_2^{-1}M_1^{-1}PX_1^{-1}=X_1P^{-1}M_1M_2M_3AM_3^{-1}M_2^{-1}M_1^{-1}PX_1^{-1}\\
\xrightarrow{\text{distance 2}}&X_1P^{-1}M_1M_2M_3JM_3^{-1}M_2^{-1}M_1^{-1}PX_1^{-1}=X_1P^{-1}M_1M_2M_3M_4JM_4^{-1}M_3^{-1}M_2^{-1}M_1^{-1}PX_1^{-1}\\
\xrightarrow{\text{distance 2}}&X_1P^{-1}M_1M_2M_3M_4A_2M_4^{-1}M_3^{-1}M_2^{-1}M_1^{-1}PX_1^{-1}\\
=&X_1P^{-1}M_1M_2M_3M_4M_5A_2M_5^{-1}M_4^{-1}M_3^{-1}M_2^{-1}M_1^{-1}PX_1^{-1}\\
=&X_1P^{-1}PXA_2X^{-1}P^{-1}PX_1^{-1}=X_2A_2X_2^{-1}.
\end{align*}
\end{proof}

\subsection{From Pivot to Pivot when $n=3$}

For $q\neq 2,n=3$, we do not have Lemma~\ref{Lem:PivotJAJ=Weak}. So we are forced to use Lemma~\ref{Lem:PivotJAJ=Strong} which requires a stronger assumption. For this end, we need a stronger version of Lemma~\ref{Lem:PivotAJAJAUnbalanced}.

As before, we consider the group $\GL_3(\F_q)$. In this group, pick $J=\begin{bmatrix}1&&1\\&1&\\&&1\end{bmatrix}$, which is a Jordan pivot matrix. We also pick $A=\begin{bmatrix}x&&\\&y&\\&&x\end{bmatrix}$ for any $x,y\in\F_q^*$ such that $x\neq y$. note that the specific values of $x,y$ does not matter, because the centralizer $\comm(A)$ will always be the same. 

\begin{lem}
\label{Lem:PivotAStrongA=All3}
For $q\neq 2$ and $n=3$, for any $X\in\GL_n(\F_q)$, we can find $R_1,R_2\in\comm(A)$ such that $R_1XR_2$ has non-zero lower left entry, and invertible lower left $2\times 2$ block.
\end{lem}
\begin{proof}
Suppose $X=\begin{bmatrix}x_{11}&x_{12}&x_{13}\\x_{21}&x_{22}&x_{23}\\x_{31}&x_{32}&x_{33}\end{bmatrix}$. Then since $X$ is invertible, the submatrix $\begin{bmatrix}x_{11}&x_{13}\\x_{31}&x_{33}\end{bmatrix}$ cannot be zero. So we can find $\begin{bmatrix}r_{11}&r_{12}\\r_{21}&r_{22}\end{bmatrix}$ and $\begin{bmatrix}r'_{11}&r'_{12}\\r'_{21}&r'_{22}\end{bmatrix}$ such that $\begin{bmatrix}r_{11}&r_{12}\\r_{21}&r_{22}\end{bmatrix}\begin{bmatrix}x_{11}&x_{13}\\x_{31}&x_{33}\end{bmatrix}\begin{bmatrix}r'_{11}&r'_{12}\\r'_{21}&r'_{22}\end{bmatrix}$ is either $\begin{bmatrix}0&1\\1&0\end{bmatrix}$ or $\begin{bmatrix}0&0\\1&0\end{bmatrix}$. 

Let $R=\begin{bmatrix}r_{11}&&r_{12}\\&1&\\r_{21}&&r_{22}\end{bmatrix}$ and $R'=\begin{bmatrix}r'_{11}&&r'_{12}\\&1&\\r'_{21}&&r'_{22}\end{bmatrix}$, then $R,R'\in\comm(A)$ and $RXR'=\begin{bmatrix}0& *& *\\ *& *& *\\ 1& *& 0\end{bmatrix}$ where the upper right entry is $0$ or $1$.

First, let us deal with the case $RXR'=\begin{bmatrix}0& a&1\\ b& c& d\\ 1& e& 0\end{bmatrix}$. 

If the lower left $2\times 2$ block is already invertible, then we are done. If the upper right $2\times 2$ block is invertible, then set $T=\begin{bmatrix}&&1\\&1&\\1&&\end{bmatrix}\in\comm(A)$. We would then have $TRXR'T$ with non-zero lower left entry and invertible lower left $2\times 2$ block as well. Since $T\in\comm(A)$ as well, we are also done.

Now suppose that the lower left $2\times 2$ block and the upper right $2\times 2$ block of $RXR'$ are both not invertible. Since the lower left block is not invertible, we must have $a,d\neq 0$. Since the upper right block is not invertible, we must have $b,e\neq 0$. Finally, since both the lower left block and the upper right block must have determinant $0$, we must have $c=ad=be$. Set $L=\begin{bmatrix}1&&\\&1&\\1&&1\end{bmatrix}\in\comm(A)$, then $RXR'L=\begin{bmatrix}1& a&1\\ b+d& c& d\\ 1& e& 0\end{bmatrix}$, and its lower left $2\times 2$ block has determinant $(b+d)e-c=de\neq 0$. So it is invertible, and we are done.

Second, let us deal with the case $RXR'=\begin{bmatrix}0& a&0\\ b& c& d\\ 1& e& 0\end{bmatrix}$. Again if the lower left $2\times 2$ block is already invertible, then we are done. Suppose that this is not the case, then $be-c=0$

Since $RXR'$ is invertible, we must have $a,d\neq 0$. Set $L=\begin{bmatrix}1&&\\&1&\\1&&1\end{bmatrix}\in\comm(A)$. If $b\neq 0$, then $LRXR'=\begin{bmatrix}0& a&0\\ b& c& d\\ 1& e+a& 0\end{bmatrix}$, and the lower left block has determinant $b(e+a)-c=ab\neq 0$. So we are done. Similarly, if $e\neq 0$, then $RXR'L=\begin{bmatrix}0& a&0\\ b+d& c& d\\ 1& e& 0\end{bmatrix}$, and the lower left block has determinant $(b+d)e-c=de\neq 0$. So we are done. 

Suppose $b=e=0$. Since the lower left block should not be invertible, we must have $c=0$ as well. So we have $RXR'=\begin{bmatrix}0& a&0\\ 0& 0& d\\ 1& 0& 0\end{bmatrix}$. Then $LRXR'L=\begin{bmatrix}0& a&0\\ d& 0& d\\ 1& a& 0\end{bmatrix}$, and the lower left block has determinant $ad\neq 0$. So we are done.

In all cases, we can find $R_1,R_2\in\comm(A)$ such that $R_1XR_2$ has non-zero lower left entry, and invertible lower left $2\times 2$ block.
\end{proof}

\begin{cor}
\label{Cor:PivotAJAJA3}
For $q\neq 2$ and $n=3$, $\GL_n(\F_q)=\comm(A)\comm(J)\comm(A)\comm(J)\comm(A)$.
\end{cor}
\begin{proof}
For any $X\in\GL_3(\F_q)$, by lemma~\ref{Lem:PivotAStrongA=All3}, we can find $R_1,R_2\in\comm(A)$ such that $R_1XR_2$ has non-zero lower left entry, and invertible lower left $2\times 2$ block. The by Lemma~\ref{Lem:PivotJAJ=Strong}, $R_1XR_2\in\comm(J)\comm(A)\comm(J)$. Therefore $X\in\comm(A)\comm(J)\comm(A)\comm(J)\comm(A)$.
\end{proof}

\begin{prop}
\label{Prop:Pivot3}
For $q\neq 2$ and $n=3$, any two pivot matrices in the projectively reduced power graph of $\GL_n(\F_q)$ will have distance at most $8$.
\end{prop}
\begin{proof}
Suppose we have two pivot matrices. Note that a $3\times 3$ pivot matrix must have exactly one eigenvalue with multiplicity $2$, and another eigenvalue with multiplicity $1$. Say they are $X_1A_1X_1^{-1}$ and $X_2A_2X_2^{-1}$ for invertible matrices $X_1,X_2$, and $A_1=\begin{bmatrix}x_1&&\\&y_1&\\&&x_1\end{bmatrix}$ and $A_2=\begin{bmatrix}x_2&&\\&y_2&\\&&x_2\end{bmatrix}$ for some $x_1,x_2,y_1,y_2\in\F_q^*$ such that $x_1\neq y_1$ and $x_2\neq y_2$. Note that we must have $\comm(A_1)=\comm(A_2)=\comm(A)$.

Set $X=X_1^{-1}X_2$. 

By Corollary~\ref{Cor:PivotAJAJA3}, we have $X=M_1M_2M_3M_4M_5$ where $M_1,M_3,M_5\in\comm(A)$ and $M_2,M_4\in\comm(J)$.

Now we have the following path.
\begin{align*}
X_1A_1X_1^{-1}=&X_1M_1A_1M_1^{-1}X_1^{-1}\\
\xrightarrow{\text{distance 2}}&X_1M_1JM_1^{-1}X_1^{-1}=X_1M_1M_2JM_2^{-1}M_1^{-1}X_1^{-1}\\
\xrightarrow{\text{distance 2}}&X_1M_1M_2AM_2^{-1}M_1^{-1}X_1^{-1}=X_1M_1M_2M_3AM_3^{-1}M_2^{-1}M_1^{-1}X_1^{-1}\\
\xrightarrow{\text{distance 2}}&X_1M_1M_2M_3JM_3^{-1}M_2^{-1}M_1^{-1}X_1^{-1}=X_1M_1M_2M_3M_4JM_4^{-1}M_3^{-1}M_2^{-1}M_1^{-1}X_1^{-1}\\
\xrightarrow{\text{distance 2}}&X_1M_1M_2M_3M_4A_2M_4^{-1}M_3^{-1}M_2^{-1}M_1^{-1}X_1^{-1}=X_1M_1M_2M_3M_4M_5A_2M_5^{-1}M_4^{-1}M_3^{-1}M_2^{-1}M_1^{-1}X_1^{-1}\\
=&X_1XA_2X^{-1}X_1^{-1}=X_2A_2X_2^{-1}.
\end{align*}

So we are done.
\end{proof}

\section{Obstructions to connectivity}
\label{sec:nDisconn}

In this section, we show several types of obstructions to the connectivity of the reduced power graphs of $\PGL_n(\F_q)$, and the projectively reduced power graph of $\GL_n(\F_q)$. As we shall see later, Proposition~\ref{prop:ConnObstJordan}, Proposition~\ref{prop:ConnObstIrred} and Proposition~\ref{prop:ConnObstDiag} are in fact the only obstructions when $q\neq 2$ and $n\geq 3$. 

For $q=2$, there are two extra types of obstructions by Proposition~\ref{prop:ConnObstDiag2} and Proposition~\ref{prop:ConnObstIrred2}. Again as we shall see later, they are the only extra obstructions for $\GL_n(\F_2)$ and $\PGL_n(\F_2)$ when $n\geq 6$. 

\begin{prop}[Jordan-type obstruction]
\label{prop:ConnObstJordan}
Consider a finite field $\F_q$ with characteristic $p$. If $2\leq n\leq p$, let $A$ be any matrix over $\F_q$ similar to the $n\times n$ matrix $\begin{bmatrix}\lambda&1&&\\ &\ddots&\ddots&\\&&\ddots&1\\ &&&\lambda\end{bmatrix}$ for some $\lambda\in\F_q^*$. Then the image of $A$ in the reduced power graph of $\PGL_n(q)$ is trapped in a connected component with $p-1$ vertices and diameter $1$, made of images of powers of $A$ that are not scalar multiples of identity.
\end{prop}
\begin{proof}
%
%
%
%
%
We want to show that, in the reduced power graph of $\PGL_n(\F_q)$, images of powers of $A$ that are not scalar multiples of identity form a connected component. Then it is enough to show that, if $B$ is a root of $A'$, a scalar multiple of a power of $A$ that is not a scalar multiple of identity, then $B$ is a multiple of a power of $A$. 

Note that since $n\leq p$, $A^p$ is a scalar multiple of identity. Suppose $B^k=zA^t$ for some $z\in\F_q^*$ and positive integers $k,t$ such that $A^t$ is not a scalar multiple of identity. Then $t$ and $p$ must be coprime. Hence $A$ is a multiple of a power of $A^t$ as well. So we may WLOG assume that $t=1$.

Suppose $B^k=zA$ for some $z\in\F_q^*$ and a positive integer $k$. Then we have $AB=BA$. Hence $B$ must be a polynomial of $A$. This implies that $B^p$ is a scalar multiple of identity, say $B^p=bI$ for some $b\in\F_q^*$. In particular, $k$ cannot be a multiple of $p$. So we can find a positive integer $k'$ such that $kk'=mp+1$ for some positive integer $m$. Then $z^{k'}A^{k'}=B^{kk'}=b^mB$. So $B$ is a scalar multiple of a power of $A$.

Hence images of powers of $A$ that are not scalar multiples of identity form a connected component. They are the images of $A,A^2,\dots,A^{p-1}$. They all corresponds to distinct vertices, and they are all powers of each other. Hence this connected component has $p-1$ vertices and diameter $1$.
\end{proof}

\begin{prop}[Irredicible obstruction]
\label{prop:ConnObstIrred}
Let $n$ be prime and $A$ be a matrix over $\F_q$ whose characteristic polynomial is irreducible. Then the image of $A$ in the reduced power graph of $\PGL_n(q)$ is trapped in a connected component with $(q^n-q)/(q-1)$ vertices, made of images of polynomials of $A$ that are not scalar multiples of identity. The diameter of this component is at most $2$.
\end{prop}
\begin{proof}
We know polynomials of $A$ form a field $\F_q[A]$ with $q^n$ elements. Pick $C$ that generate $\F_q[A]^*$ multiplicatively. Then all non-zero polynomials of $A$ are powers of $C$. Clearly $C$ is also similar to a companion matrix of an irreducible polynomial of degree $n$.

It is enough to show that, if $B$ is a root of $C'$, a scalar multiple of a power of $C$ that is not a scalar multiple of identity, then $B$ is a scalar multiple of a power of $C$. Then this would imply that images of polynomials of $A$ (or equivalently, powers of $C$) that are not scalar multiples of identity form a connected component, where everyone is connected to $C$. Hence it will have $(q^n-q)/(q-1)$ vertices and diameter at most $2$.

Suppose $B^k=zC^t$ for some $z\in\F_q^*$ and positive integer $k,t$ such that $C^t$ is not a scalar multiple of identity. Note that by Lemma~\ref{Lem:CompanionDecomp}, the generalized Jordan canonical form of $C^t$ is made of identical diagonal blocks that are the same companion matrix of some irreducible polynomial. Since $n$ is prime, either these blocks are $1\times 1$ which implies that $C^t$ is a scalar multiple of identity, or the block is $n\times n$. Only the latter case is possible. So $C^t$ is similar to a companion matrix of an irreducible polynomial.

Now we have $BC^t=C^tB$. By Lemma~\ref{Lem:CompanionComm}, $B$ must be a polynomial of $C^t$. In particular, $B\in\F_q[A]$. Since we know $B^k=zC$, we know that $B$ is invertible and hence non-zero. Therefore $B\in\F_q[C]^*$, and it must be a power of $C$. So we are done.
\end{proof}

\begin{lem}
\label{Lem:CompanionComm}
For $n\geq 2$, consider an $n\times n$ companion matrix $A=\begin{bmatrix}\bs e_2&\bs e_3&\dots&\bs e_n&\bs v\end{bmatrix}$ where $\bs e_1,\dots,\bs e_n$ are the standard basis for $\F_q^n$ and $\bs v$ is an arbitrary vector in $\F_q^n$. If $AB=BA$ for some $n\times n$ matrix $B$, then $B$ is a polynomial of $A$.
\end{lem}
\begin{proof}
Let $B=\begin{bmatrix}\bs b_1&\dots&\bs b_n\end{bmatrix}$. Then we have 
\[BA=\begin{bmatrix}B\bs e_2&B\bs e_3&\dots&B\bs e_n&B\bs v\end{bmatrix}=\begin{bmatrix}\bs b_2&\dots&\bs b_n&B\bs v\end{bmatrix},\] 
\[AB=\begin{bmatrix}A\bs b_1&\dots&A\bs b_n\end{bmatrix}.\] 
Hence $AB=BA$ implies that $\bs b_{i+1}=A\bs b_i$ for $1\leq i\leq n-1$. So we have 
\[B=\begin{bmatrix}\bs b_1&A\bs b_1&\dots&A^{n-1}\bs b_1\end{bmatrix}.\]

Now we set
\[B_i=\begin{bmatrix}\bs e_i&A\bs e_i&\dots&A^{n-1}\bs e_i\end{bmatrix}=A^{i-1}.\]

Suppose $\bs b_1=\sum x_i\bs e_i$ for coefficients $x_1,\dots,x_n\in\F_q$. Then $B=\sum x_i B_i=\sum x_iA^{i-1}$. Hence $B$ is a polynomial of $A$.
\end{proof}

\begin{prop}[Diagonalizable obstruction]
\label{prop:ConnObstDiag}
Suppose $q>2$ is a power of $2$ with $q-1$ prime, and $2\leq n<q$. Let $A$ be an $n\times n$ diagonalizable matrix with $n$ distinct non-zero eigenvalues over $\F_q$. Then the image of $A$ in the reduced power graph of $\PGL_n(q)$ is trapped in a connected component with $q-2$ vertices and diameter $1$, made of images of non-identity powers of $A$.
\end{prop}
\begin{proof}
It is enough to show that, if $B$ is a root of $A'$, a scalar multiple of a power of $A$ that is not a scalar multiple of identity, then $B$ is a scalar multiple of a power of $A$. Then powers of $A$ that are not scalar multiples of identity would form a connected component in the reduced power graph of $\PGL_n(q)$ with vertices represented by $A,\dots,A^{q-2}$. None of these matrices can be scalar multiples of each other. They all have order $q-1$, which is prime, so they are all powers of each other. So this is a connected component with $q-2$ vertices and diameter $1$.

Suppose $B^k=zA^t$ for some $z\in\F^*$ and positive integer $k,t$ such that $A^t$ is not a scalar multiple of identity. Note that if $A^t$ is not a scalar multiple of identity, then $t$ must be coprime to the prime number $q-1$. Hence $A^t$ and $A$ are powers of each other, and $A^t$ also has $n$ distinct non-zero eigenvalues over $\F_q$. Then $BA^t=A^tB$ implies that $A^t,B$ are simultaneously diagonalizable over $\F_q$. Suppose $A^t$ has eigenvalues $a_1,\dots,a_n$ and $B$ has corresponding eigenvalues $b_1,\dots,b_n$ such that $b_i^k=za_i$. Note that $q-1$ is prime, so the multiplicative group $\F_q^*$ is cyclic of prime order. Hence if all $b_i^k$ are distinct, we must have $k$ coprime to $q-1$. Let $k'$ be a positive integer such that $kk'$ is $1$ modulus $q-1$. Then $b_i=z^{k'}a_i^{k'}$. Hence we have $B=z^{k'}A^{tk'}$. So we are done.
\end{proof}

Finally, we discuss two special type of obstructions to connectivity when $q=2$. Note that in this case, $\PGL_n(2)=\GL_n(2)$.

The following obstruction is essentially a variant of the obstruction above.

\begin{prop}[Quasi-diagonalizable obstruction when $q=2$]
\label{prop:ConnObstDiag2}
Suppose we have a prime $p_0=2^{p_1}-1$ (i.e., Mersenne prime) for some positive integer $p_1$. 
Suppose $n\leq p_0$. 
Let $A$ be an $n\times n$ non-identity matrix such that $A^{p_0}=I$, and its minimal polynomial equal to its characteristic polynomial. Then the $A$ in the reduced power graph of $\GL_n(2)$ is trapped in a connected component with $p_0-1$ vertices and diameter $1$, made of non-identity powers of $A$.

Note that in this case, $p_1$ must be a prime factor of $n$ or $n-1$.
\end{prop}
\begin{proof}
Set $q=p_0+1$. Then $q$ is a power of $2$, so we can pick an extension $\F_q$ of $\F_2$. Let us treat $A$ as a matrix over $\F_{q}$ instead. 

Then $A^{q-1}=I$ implies that $A$ is diagonalizable over $\F_q$. And since its minimal polynomial is the same as its characteristic polynomials, it has $n$ distinct non-zero eigenvalues over $\F_q$, and $2\leq n<q$, and $q-1$ is a prime. So we are in the case of Proposition~\ref{prop:ConnObstDiag}. So $A$ is trapped in a connected component even in the reduced power graph of $\PGL_n(q)$, made of non-identity powers of $A$. Therefore in the subgraph made of matrices in $\PGL_n(2)$, $A$ is also trapped in the same connected component.

For the last statement, if $p_0=2^{p_1}-1$ is a prime, then it is a Mersenne prime. Thus $p_1$ is also prime. 

Now if $A$ has multiplicative order $p_0$, then by Lemma~\ref{Lem:NonJordayTypeClassify}, the generalized Jordan canonical form of $A$ is block diagonal with diagonal blocks $C_1,\dots,C_t$ of companion matrices for irreducible polynomials. Since $A$ has multiplicative order $p_0$, each $C_i$ must have multiplicative order dividing $p_0$. So either $C_i$ is $1\times 1$, or $C_i$ is $p_1\times p_1$ by Corollary~\ref{Cor:CompanionSameOrderSameSize}.

Since the minimal polynomial and characteristic polynomial of $A$ are the same, all $C_i$ must be distinct. Hence there is at most one block of size $1\times 1$, and all other blocks must be $p_1\times p_1$. Therefore $n$ or $n-1$ is a multiple of $p_1$.
\end{proof}

The second one is a variant of Proposition~\ref{prop:ConnObstIrred}.

\begin{prop}[Extra irredicible obstruction when $q=2$]
\label{prop:ConnObstIrred2}
Let $n-1$ be prime and $A$ be an $n\times n$ matrix over $\F_2$ whose characteristic polynomial has an irreducible factor of degree $n-1$. Then the image of $A$ in the reduced power graph of $\PGL_n(2)$ is trapped in a connected component with $2^{n-1}-2$ vertices, made of non-identity powers of a matrix $C\in\PGL_n(2)$. The diameter of this component is at most $2$.
\end{prop}
\begin{proof}
Consider the generalized Jordan canonical form of $A$, we must have $A=X\begin{bmatrix}A'&\\&1\end{bmatrix}X^{-1}$ for invertible $X$, where $A'$ is a companion matrix of an irreducible polynomial of degree $n-1$. 

We know polynomials of $A'$ form a field $\F_q[A']$ with $2^{n-1}$ elements. Pick $C'$ that generate $\F_q[A']^*$ multiplicatively. Then $A'$ is a power of $C'$. Then all non-zero polynomials of $A'$ are powers of $C'$.
Let $C=X\begin{bmatrix}C'&\\&1\end{bmatrix}X^{-1}$, then $A$ is similar to a power of $C$.

It is enough to show that, if $B$ is a root of $C''$, a scalar multiple of a power of $C$ that is not a scalar multiple of identity, then $B$ is a scalar multiple of a power of $C$. Then this would imply that non-identity powers of $C$ form a connected component, where everyone is connected to $C$. Since $C$ has multiplicative order $2^{n-1}-1$, this connected component will have $2^{n-1}-2$ vertices and diameter at most $2$.

Suppose $B^k=C^t$ for some positive integers $k,t$ such that $C^t=X\begin{bmatrix}(C')^t&\\&1\end{bmatrix}X^{-1}$ is not identity. In particular, $(C')^t$ is not identity. By Lemma~\ref{Lem:CompanionDecomp}, the generalized Jordan canonical form of $(C')^t$ is made of identical diagonal blocks that are the same companion matrix of some irreducible polynomial. Since $n-1$ is prime, either these blocks are $1\times 1$ which (over $\F_2$) implies that $C^t$ is identiy, or there is only one $(n-1)\times (n-1)$ block. Only the latter case is possible. So $(C')^t$ is similar to a companion matrix of an irreducible polynomial.

Now we have $BC^t=C^tB$. Since $(C')^t$ has no eigenvalue $1$, this implies that $B=X\begin{bmatrix}B'&\\&1\end{bmatrix}X^{-1}$ for some $(n-1)\times(n-1)$ matrix $B'$ such that $B'(C')^t=(C')^tB'$. (Here the lower right entry has to be $1$, because our field is $\F_2$ and there is no other non-zero element.) Then by Lemma~\ref{Lem:CompanionComm}, since $(C')^t$ is similar to a companion matrix, $B'$ must be a polynomial of $(C')^t$. In particular, $B'\in\F_q[C']$. But $B$ is non-zero due to $B^k=C$, therefore $B'$ is a power of $C'$. Consequently, $B$ is a power of $C$. So we are done.
\end{proof}

Finally, for the sake of completion, here is the criteria for when the diameter in Proposition~\ref{prop:ConnObstIrred} and Proposition~\ref{prop:ConnObstIrred2} is one.

\begin{prop}
In Proposition~\ref{prop:ConnObstIrred}, the diameter of each connected component is $1$ if and only if $q-1$ is a prime power and $\frac{q^n-1}{q-1}$ is a prime. Similarly, in Proposition~\ref{prop:ConnObstIrred2}, the diameter of each connected component is $1$ if and only if $2^{n-1}-1$ is a prime.
\end{prop}
\begin{proof}
In either case, the connected component is made of powers of the same element. So this follows by the property of cyclic groups below in Lemma~\ref{Lem:cyclicForObstructionDiameter1or2}.
\end{proof}

\begin{lem}
\label{Lem:cyclicForObstructionDiameter1or2}
Let $q$ be a power of a prime $p$, and pick any positive integer $n\geq 3$. In a cyclic group $G$ with $q^n-1$ elements, if we can find two subgroups $H,K$ such that neither contains the other, and neither is the subgroup whose order divides $q-1$, then $q-1$ is a prime power and $\frac{q^n-1}{q-1}$ is a prime, and this condition is necessary and sufficient.
\end{lem}
\begin{proof}
Let $x$ be a generator of $G$. Let us write the group operation in $G$ multiplicatively.

Suppose $q-1$ is not a prime power. Then we can find distinct prime factors $p_1,p_2$ of $q-1$, and the elements $x^{p_1},x^{p_2}$ would respectively generate subgroups $H,K$ that does not contain each other. Since $p_i<q-1\leq\frac{q^n-1}{q-1}$ for $i=1,2$, neither $H$ nor $K$ can have order dividing $q-1$.

Now suppose $q-1$ is a power of a prime $p_1$, but $\frac{q^n-1}{q-1}$ is not a prime. Pick any prime factor $p_2$ of $\frac{q^n-1}{q-1}$. If $p_1=p_2$, then $q^n-1$ is a prime power. By Lemma~\ref{Lem:qnPrimeEvasionQ}, this means $q=3$ and $n=2$. So this cannot happen when $n\geq 3$. Now suppose $p_1\neq p_2$. Then again $x^{p_1},x^{p_2}$ would respectively generate subgroups $H,K$ that does not contain each other. Now obviously $p_1\leq q-1\leq\frac{q^n-1}{q-1}$, but since $\frac{q^n-1}{q-1}$ is not prime, we have $p_1<\frac{q^n-1}{q-1}$. At the same time, since $\frac{q^n-1}{q-1}$ is not prime, its factor $p_2$ must also be strictly less than $\frac{q^n-1}{q-1}$. So again neither $H$ nor $K$ can have $q-1$ elements.

Finally, suppose $q-1$ is a power of a prime $p_1$, and $p_2=\frac{q^n-1}{q-1}$ is a prime. Then the order of $G$ is $p_2$ times a power of $p_1$. Then for any two subgroups $H,K$ of $G$ whose order do not divide $q-1$, then $p_2$ must divide both of their orders. Hence their index must be powers of $p_1$. In a cyclic group, this means one subgroup must contain the other.
\end{proof}

%
%
%
%
%

\section{The connectivity of the projectively reduced power graphs of $\GL_n(q)$}
\label{sec:nConn}

\subsection{Overview}

We want to study the connectivity of the projectively reduced power graphs of $\GL_n(\F_q)$. Our strategy is to connect most element of $\GL_n(\F_q)$ to (Jordan) pivot matrices, with exceptions outlined in section~\ref{sec:nDisconn}. Since we are dealing with the projectively reduced power graph, we must avoid the center of $\GL_n(\F_q)$, i.e., scalar multiples of the identity matrix. Hence we make the following definition.

\begin{defn}
For a matrix $A\in\GL_n(\F_q)$, its projective order is the multiplicative order of the image of $A$ in $\PGL_n(\F_q)$.
\end{defn}

This concept is useful to understand the generalized Jordan canonical form of $A$.

\begin{lem}
\label{Lem:NonJordayTypeClassify}
Let $q$ be a power of a prime $p$, and let $n\geq 2$. If a matrix $A\in\GL_n(\F_q)-Z(\GL_n(\F_q))$ has projective order coprime to $p$, then its generalized Jordan canonical form is a block diagonal matrix where each diagonal block is either $1\times 1$ or some companion matrix of an irreducible polynomial.
\end{lem}
\begin{proof}
Suppose $A^k=\lambda I$ for some $k$ coprime to $p$ and $\lambda\in\F_q^*$. Then the minimal polynomial of $A$ is a factor of $f(x)=x^k-\lambda$. Since $k$ is coprime to $p$, the polynomial $f(x)$ is separable. Hence the minimal polynomial of $A$ is separable, and our statement follows.
\end{proof}

Given any matrix, by raising it to some power, we can obtian a matrix whose projective order is a prime number. In fact, we can also require that this matrix also has prime-power multiplicative order.

\begin{lem}
\label{Lem:RaisingAtoProjectivePrime}
If $p_0$ is a prime number dividing the projective order of a matrix $A\in\GL_n(\F_q)$, then $A$ has a power $A'$ whose multiplicative order is a power of $p_0$, and projective order is exactly $p_0$.
\end{lem}
\begin{proof}
Let the projective order of $A$ be $mp_0^k$ for a positive integer $m$ coprime to $p_0$, and a positive integer $k$. Let the multiplicative order of $A$ be $m'p_0^{k'}$ for a positive integer $m$ coprime to $p_0$, and a positive integer $k$. Note that $m$ must divide $m'$ and we must have $1\leq k\leq k'$.

Now set $A'=A^{m'p_0^{k-1}}$. Then $A'$ has multiplicative order $p_0^{k'-k+1}$, and projective order $p_0$ exactly.
\end{proof}

We can now discuss how a matrix in the projectively reduced graph of $\GL_n(\F_q)$ could have a path to a pivot or Jordan pivot matrix.

When $q\neq 2$ and $n\geq 3$, the path towards a pivot matrix is outlined in Figure~\ref{Fig:qNot2Conn}. (Here $A$ is a generic matrix in $\GL_n(\F_q)$ and not a scalar multiple of identity. $p$ is the characteristic of the field $\F_q$.) 

When $q=2$ and $n\geq 6$, the path towards a Jordan pivot matrix is outlined in Figure~\ref{Fig:q=2Conn}.

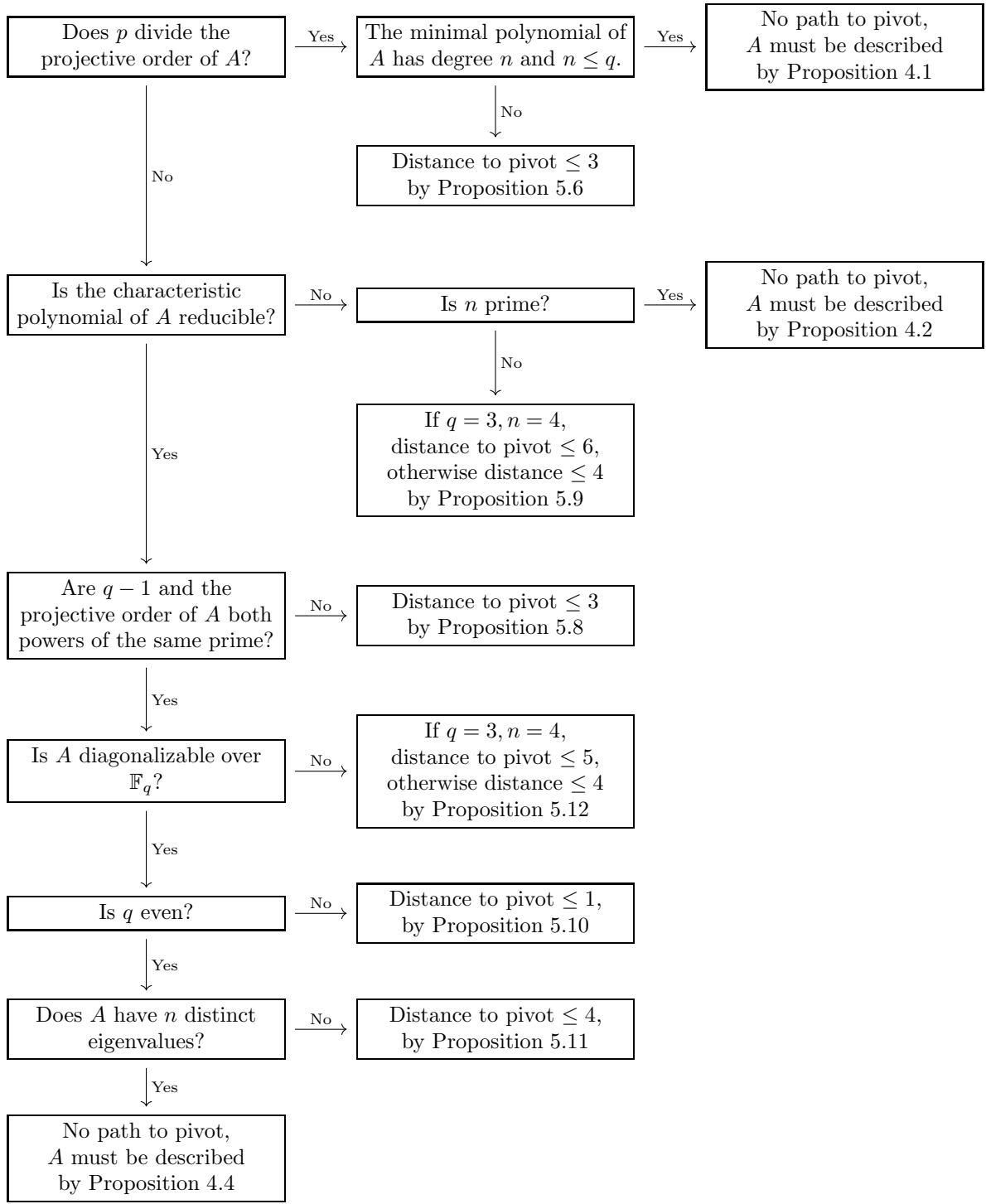
\begin{figure}
\centering
 \begin{tikzcd}[
ampersand replacement=\&]
\fbox{\begin{minipage}{12em}
\centering
Does $p$ divide the projective order of $A$?
\end{minipage}} 
\arrow[r,"\text{Yes}"]
\arrow[dd,"\text{No}"]
\&
\fbox{\begin{minipage}{12em}
\centering
The minimal polynomial of $A$ has degree $n$ and $n\leq q$.
\end{minipage}}
\arrow[r,"\text{Yes}"]
\arrow[d,"\text{No}"]
\&
\fbox{\begin{minipage}{12em}
\centering
No path to pivot,\\
$A$ must be described\\
by Proposition~\ref{prop:ConnObstJordan}
\end{minipage}} 
\\
\&
\fbox{\begin{minipage}{12em}
\centering
Distance to pivot $\leq 3$\\
by Proposition~\ref{Prop:JordayTypePower}
\end{minipage}} 
\&
\\
\fbox{\begin{minipage}{12em}
\centering
Is the characteristic polynomial of $A$ reducible?
\end{minipage}} 
\arrow[r,"\text{No}"]
\arrow[dd,"\text{Yes}"]
\&
\fbox{\begin{minipage}{12em}
\centering
Is $n$ prime?
\end{minipage}}
\arrow[r,"\text{Yes}"]
\arrow[d,"\text{No}"]
\&
\fbox{\begin{minipage}{12em}
\centering
No path to pivot,\\
$A$ must be described\\
by Proposition~\ref{prop:ConnObstIrred}
\end{minipage}}
\\
\&
\fbox{\begin{minipage}{12em}
\centering
If $q=3,n=4$,\\
distance to pivot $\leq 6$,\\
otherwise distance $\leq 4$\\
by Proposition~\ref{Prop:NonJordayTypeIrredChar}
\end{minipage}} 
\&
\\
\fbox{\begin{minipage}{12em}
\centering
Are $q-1$ and the projective order of $A$ both powers of the same prime?
\end{minipage}} 
\arrow[r,"\text{No}"]
\arrow[d,"\text{Yes}"]
\&
\fbox{\begin{minipage}{12em}
\centering
Distance to pivot $\leq 3$\\
by Proposition~\ref{Prop:NonJordayTypePrimeEvasion}
\end{minipage}}
\&
\\
\fbox{\begin{minipage}{12em}
\centering
Is $A$ diagonalizable over $\F_q$?
\end{minipage}} 
\arrow[r,"\text{No}"]
\arrow[d,"\text{Yes}"]
\&
\fbox{\begin{minipage}{12em}
\centering
If $q=3,n=4$,\\
distance to pivot $\leq 5$,\\
otherwise distance $\leq 4$\\
by Proposition~\ref{Prop:NonJordayTypeBothPowerSamePrime}
\end{minipage}}
\&
\\
\fbox{\begin{minipage}{12em}
\centering
Is $q$ even?
\end{minipage}} 
\arrow[r,"\text{No}"]
\arrow[d,"\text{Yes}"]
\&
\fbox{\begin{minipage}{12em}
\centering
Distance to pivot $\leq 1$,\\
by Proposition~\ref{Prop:NonJordayTypeDiagOdd}
\end{minipage}}
\&
\\
\fbox{\begin{minipage}{12em}
\centering
Does $A$ have $n$ distinct eigenvalues?
\end{minipage}} 
\arrow[r,"\text{No}"]
\arrow[d,"\text{Yes}"]
\&
\fbox{\begin{minipage}{12em}
\centering
Distance to pivot $\leq 4$,\\
by Proposition~\ref{Prop:NonJordayTypeDiagEven}
\end{minipage}}
\&
\\
\fbox{\begin{minipage}{12em}
\centering
No path to pivot,\\
$A$ must be described\\
by Proposition~\ref{prop:ConnObstDiag}
\end{minipage}} 
\&
\&
\end{tikzcd}
\caption{Path to pivot matrices when $q\neq 2$}
\label{Fig:qNot2Conn}
\end{figure}

\begin{figure}
\centering
 \begin{tikzcd}[
ampersand replacement=\&]
\fbox{\begin{minipage}{12em}
\centering
Is the multiplicative order of $A$ even?
\end{minipage}} 
\arrow[r,"\text{Yes}"]
\arrow[d,"\text{No}"]
\&
\fbox{\begin{minipage}{12em}
\centering
Distance to pivot $\leq 5$\\
by Proposition~\ref{Prop:JordayTypePower2}
\end{minipage}}
\&
\\
\fbox{\begin{minipage}{12em}
\centering
Does the characteristic polynomial of $A$ has an irreducible factor of degree $n$ or $n-1$?
\end{minipage}} 
\arrow[r,"\text{Yes}"]
\arrow[dd,"\text{No}"]
\&
\fbox{\begin{minipage}{12em}
\centering
Is the degree of this irreducible factor a composite number?
\end{minipage}} 
\arrow[r,"\text{Yes}"]
\arrow[d,"\text{No}"]
\&
\fbox{\begin{minipage}{12em}
\centering
Distance to pivot $\leq 6$\\
by Proposition~\ref{Prop:NonJordayTypeIrredFactorN2}
\end{minipage}} 
\\
\&
\fbox{\begin{minipage}{12em}
\centering
No path to pivot,\\
$A$ must be described\\
by Proposition~\ref{prop:ConnObstIrred}\\
or Proposition~\ref{prop:ConnObstIrred2}
\end{minipage}}
\&
\\
\fbox{\begin{minipage}{12em}
\centering
Is the multiplicative order of $A$ a composite number?
\end{minipage}} 
\arrow[r,"\text{Yes}"]
\arrow[d,"\text{No}"]
\&
\fbox{\begin{minipage}{12em}
\centering
Distance to pivot $\leq 5$\\
by Proposition~\ref{Prop:NonJordayTypeNonPrimeOrder2}
\end{minipage}} 
\&
\\
\fbox{\begin{minipage}{12em}
\centering
Is one plus the multiplicative order of $A$ a power of two?
\end{minipage}} 
\arrow[r,"\text{No}"]
\arrow[d,"\text{Yes}"]
\&
\fbox{\begin{minipage}{12em}
\centering
Distance to pivot $\leq 4$\\
by Proposition~\ref{Prop:NonJordayTypePrimeOrderNonPrimeBlock2}
\end{minipage}} 
\&
\\
\fbox{\begin{minipage}{12em}
\centering
Does $A$ have the same minimal polynomial and characteristic polynomial?
\end{minipage}} 
\arrow[r,"\text{No}"]
\arrow[d,"\text{Yes}"]
\&
\fbox{\begin{minipage}{12em}
\centering
Distance to pivot $\leq 6$\\
by Proposition~\ref{Prop:NonJordayTypePrimeOrderMersennePrimeBlock2}
\end{minipage}} 
\&
\\
\fbox{\begin{minipage}{12em}
\centering
No path to pivot,\\
$A$ must be described\\
by Proposition~\ref{prop:ConnObstDiag2}
\end{minipage}}
\&
\&
\end{tikzcd}
\caption{Path to pivot matrices when $q=2$}
\label{Fig:q=2Conn}
\end{figure}
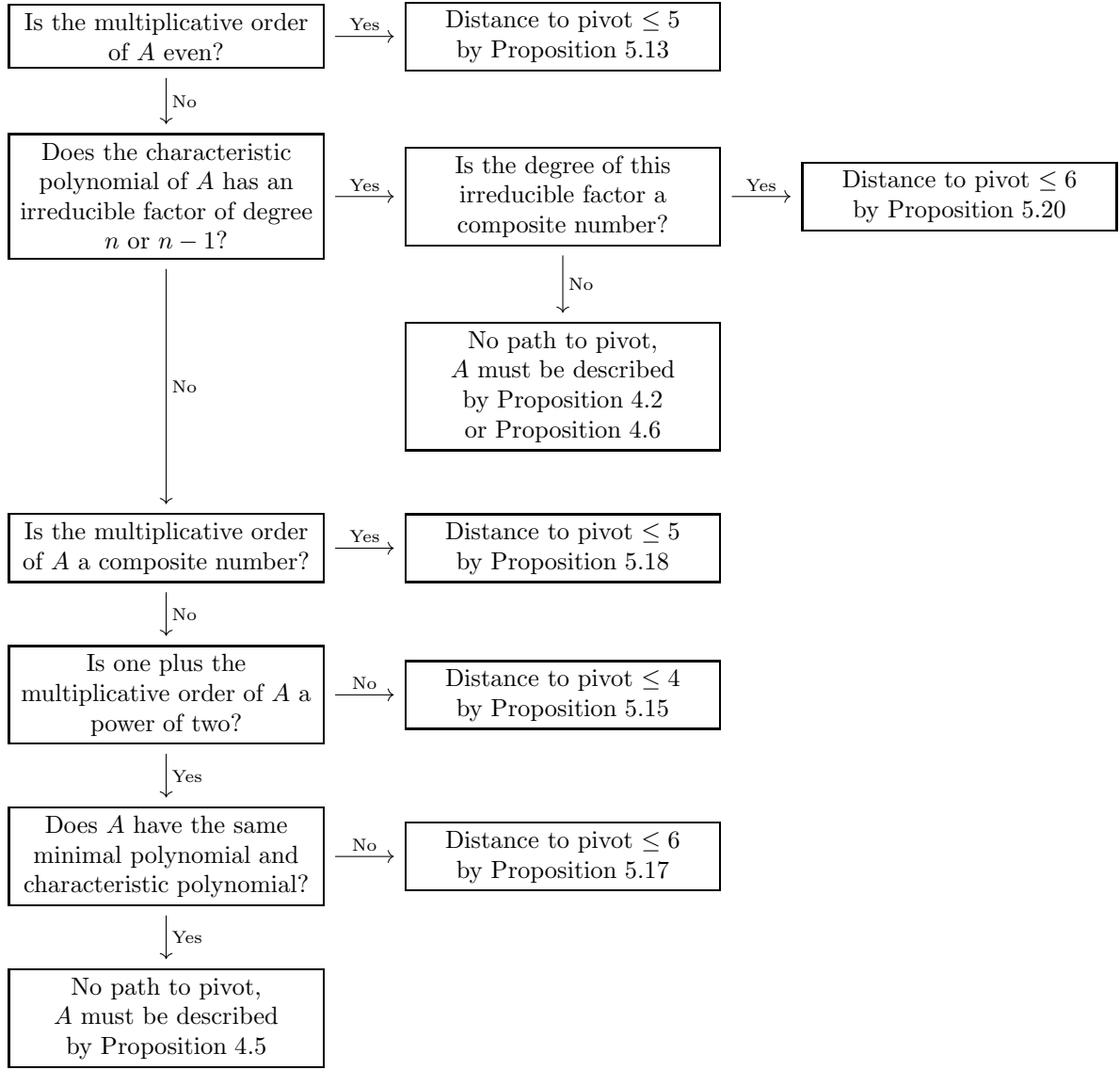

In particular, we have the following results.

\begin{prop}
\label{Prop:PathToPivotAllCases}
We have the following connection pattern in the projectively reduced power graph for $\GL_n(\F_q)$.
\begin{enumerate}
\item If $q=2$ and $n\geq 6$, then all matrices in $\GL_n(\F_q)-Z(\GL_n(\F_q))$ must either have a distance at most $6$ to a Jordan pivot matrix, or have no path to pivot matrices.
\item If $q=3$ and $n=4$, then all matrices in $\GL_n(\F_q)-Z(\GL_n(\F_q))$ must either have a distance at most $6$ to a pivot matrix, or have no path to pivot matrices.
\item If $n\geq 3$ and we are not in the cases above, then all matrices in $\GL_n(\F_q)-Z(\GL_n(\F_q))$ must either have a distance at most $4$ to a pivot matrix, or have no path to pivot matrices.
\item If $q\neq 2$ and $n\geq 3$, and a matrix in $\GL_n(\F_q)-Z(\GL_n(\F_q))$ has no path to a pivot matrix, then it must be as described in Proposition~\ref{prop:ConnObstJordan}, Proposition~\ref{prop:ConnObstIrred}, and Proposition~\ref{prop:ConnObstDiag}.
\item If $q=2$ and $n\geq 6$, and a matrix in $\GL_n(\F_q)-Z(\GL_n(\F_q))$ has no path to a pivot matrix, then it must be as described in Proposition~\ref{prop:ConnObstIrred}, Proposition~\ref{prop:ConnObstDiag}, Proposition~\ref{prop:ConnObstDiag2}, and Proposition~\ref{prop:ConnObstIrred2}.
\end{enumerate}
\end{prop}

\subsection{When $q\neq 2$}

We first deal with the case when the characteristic $p$ of $\F_q$ divides the projective order of $A$.

\begin{lem}
\label{Lem:JordayTypePower}
Let $q$ be a power of a prime $p$, and let $n\geq 2$. For an $n\times n$ matrix $A$ over $\F_q$, if $p$ divides the projective order of $A$, then $A$ has a power $A'$ whose minimal polynomial is $(x-1)^k$ for some $2\leq k\leq p$.
\end{lem}
\begin{proof}
By Lemma~\ref{Lem:RaisingAtoProjectivePrime}, $A$ has a power $A'$ whose projective order is $p$, and its multiplicative order is a power of $p$.

Suppose $(A')^p=\lambda I$ for some $\lambda\in\F_q^*$. Then the minimal polynomial of $A'$ must divide the polynomial $f(x)=x^p-\lambda$. Since $p$ is the characteristic of the field $\F_q$, we can find $\mu\in\F_q$ such that $\mu^p=\lambda$. Then $f(x)=x^p-\mu^p=(x-\mu)^p$. So the minimal polynomial of $A'$ is a power of $x-\mu$ with degree at most $p$.

Furthermore, since the multiplicative order of $A'$ is a power of $p$, its eigenvalue $\mu$ must also have multiplicative order a power of $p$ in the group $\F_q^*$. But since the group $\F_q^*$ has $q-1$ elements, which is coprime to $p$, therefore $\mu=1$. So we are done.
\end{proof}

\begin{prop}
\label{Prop:JordayTypePower}
For any power $q\neq 2$ of a prime $p$, and let $n\geq 2$, suppose $p$ divides the projective order of a matrix $A\in\GL_n(\F_q)$. Then either $A$ has distance at most $3$ to a pivot matrix, or $A$ has not path to any pivot matrix, and $A$ must be as described in Proposition~\ref{prop:ConnObstJordan}.
\end{prop}
\begin{proof}
By Lemma~\ref{Lem:JordayTypePower}, $A$ has a power $A'$ whose minimal polynomial is $(x-1)^k$ for some $2\leq k\leq p$. So in the generalized Jordan canonical form of $A'$, all diagonal blocks are of the form $\begin{bmatrix}1&1&&\\&\ddots&\ddots&\\&&\ddots&1\\&&&1\end{bmatrix}$, and the size of each block is at most $p\times p$. 

Suppose $A'$ has at least two such blocks. Say $A'=X\begin{bmatrix}J&\\&A''\end{bmatrix}X^{-1}$ for some invertible $X$, where $J=\begin{bmatrix}1&1&&\\&\ddots&\ddots&\\&&\ddots&1\\&&&1\end{bmatrix}$ is at most $p\times p$, and $A''$ also has minimal polynomial $(x-1)^{k'}$ for some $2\leq k'\leq p$. In particular, $J^p$ and $(A'')^p$ are both identity matrices (maybe with different size).

Pick any $x\neq 1\in\F_q^*$. Consider $B=X\begin{bmatrix}xJ&\\&A''\end{bmatrix}X^{-1}$. Then $B^q=X\begin{bmatrix}x^qI&\\&I\end{bmatrix}X^{-1}=X\begin{bmatrix}xI&\\&I\end{bmatrix}X^{-1}$ is a pivot matrix. On the other hand, $B^{(q-1)^2}=X\begin{bmatrix}J&\\&A''\end{bmatrix}X^{-1}=A'$. So a power of a root of $A'$ is a pivot matrix. So $A$ has distance at most $3$ to a pivot matrix.

Now suppose $A'$ is similar to a single block $\begin{bmatrix}1&1&&\\&\ddots&\ddots&\\&&\ddots&1\\&&&1\end{bmatrix}$, and the size of this block is at most $p\times p$. Then immediately we have $n\leq p$, and $A'$ must be as described in Proposition~\ref{prop:ConnObstJordan}. Since $A$ and $A'$ must be in the same connected component, $A$ must also be as described in Proposition~\ref{prop:ConnObstJordan}.
\end{proof}

As we can see, when $q\neq 2$ and $p$ divides the projective order of $A$, then the only possible obstruction is from Proposition~\ref{prop:ConnObstJordan}.

we now deal with matrices whose projective order is coprime to the characteristic of the base field. The path to a pivot matrix would usually go through a middle matrix, as described in Lemma~\ref{Lem:NonJordayTypePrimeEvasion}.

\begin{lem}
\label{Lem:NonJordayTypePrimeEvasion}
Let $q\neq 2$ be a power of a prime $p$, and let $n\geq 2$. Suppose a matrix $A\in\GL_n(\F_q)$ has prime projective order $p_1\neq p$, and its multiplicative order is also a power of $p_1$. Furthermore, suppose $q-1$ is not a power of $p_1$, say $p_0$ is another prime factor of $q-1$. Finally, suppose that the characteristic polynomial of $A$ is reducible. Then $A$ has distance at most $2$ to a pivot matrix with multiplicative order $p_0$.
\end{lem}
\begin{proof}
Since we have $p_1\neq p$, by Lemma~\ref{Lem:NonJordayTypeClassify}, $A=X\begin{bmatrix}C_1&&\\ &\ddots&\\&&C_t\end{bmatrix}X^{-1}$ for some invertible $X$, where $C_1,\dots, C_t$ are companion matrices to various irreducible polynomials over $\F_q$. Since the characteristic polynomial of $A$ is not irreducible, we know that $t\geq 2$. By Corollary~\ref{Cor:CompanionSameOrderSameSize}, we can find a positive integer $k$ depending only on $p_1$ and $q$, such that all $C_i$ are either $1\times 1$ or $k\times k$. 

For each $i$, let $m_i$ be the multiplicative order of the block $C_i$. Then $m_i$ must divide the multiplicative order of $A$, hence it is a power of $p_1$. 

Since $q-1$ is not a power of $p_1$, we can find a prime factor $p_0$ of $q-1$ distinct from $p_1$. Then $p_0$ is also a factor of $q^k-1$, and it is coprime to $m_i$. If $C_i$ is $k\times k$, then by Lemma~\ref{Lem:CompanionOrderBasic}, $m_i$ is a factor of $q^k-1$. Therefore, $p_0m_i$ is a factor of $q^k-1$. By Lemma~\ref{Lem:CompanionRoot}, this means $C_i$ must have a $p_0$-th root $C'_i$ whose multiplicative order is $p_0m_i$. Similarly, if $C_i$ is $1\times 1$, then $p_0p_1$ is a factor of $q-1$, and again $C_i$ must have a $p_0$-th root $C'_i$ whose multiplicative order is $p_0m_i$. 

Since $p_0$ divides $q-1$, we can find $x\neq 1\in\F_q^*$ such that $x^{p_0}=1$. Let $m$ be the multiplicative order of $A$ which is the least common multiples of all $m_i$, and hence coprime to $p_0$. Then $(C'_i)^{m}$ will have multiplicative order $p_0$, a factor of $q-1$. By Corollary~\ref{Cor:CompanionPower}, $(C'_i)^{m}=x^{d_i}I$ for some integer $d_i$. 

Since $m$ is coprime to $p_0$, we can pick a positive integer $m'$ such that $mm'$ is $1$ modulus $p_0$. Then we have a path in the projectively reduced power graph of $\GL_n(\F_q)$.
\begin{align*}
A=&X\begin{bmatrix}C_1&&\\ &\ddots&\\&&C_t\end{bmatrix}X^{-1}\\
\xrightarrow{\text{$p_0$-th root}}&
X\begin{bmatrix}x^{(1-d_1)m'}C'_1&&&\\ &x^{-d_2m'}C'_2&&\\ &&\ddots&\\&&&x^{-d_tm'}C'_t\end{bmatrix}X^{-1}\\
\xrightarrow{\text{$m$-th power}}&
X\begin{bmatrix}xI&&&\\ &I&&\\ &&\ddots&\\&&&I\end{bmatrix}X^{-1}.
\end{align*}

The last matrix here is a pivot matrix, so we are done.
\end{proof}

This immediately shows that many matrices are connected to pivot matrices by some short path.

\begin{prop}
\label{Prop:NonJordayTypePrimeEvasion}
Let $q\neq 2$ be a power of a prime $p$, and let $n\geq 2$. Suppose the characteristic polynomial of a matrix $A\in\GL_n(\F_q)$ is reducible, and we can find a prime $p_1$ dividing the projective order of $A$, such that $q-1$ is not a power of $p_1$. Then $A$ has distance at most $3$ to a pivot matrix.
\end{prop}
\begin{proof}
By Lemma~\ref{Lem:RaisingAtoProjectivePrime}, $A$ has a power $A'$ with projective order $p_1$, and multiplicative order a power of $p_1$. Since the characteristic polynomial of $A$ is reducible, the characteristic polynomial of $A'$ must also be reducible. Then we are done by Lemma~\ref{Lem:NonJordayTypePrimeEvasion} above.
\end{proof}

We are left with only two situations: when the characteristic polynomial of $A$ is irreducible, or when the projective order of $A$ and $q-1$ are both powers of the same prime.

If the characteristic polynomial of an $n\times n$ matrix $A$ is irreducible, and $n$ is prime, then this is exactly the situation as described in Proposition~\ref{prop:ConnObstIrred}. If $n$ is not prime, then the situation is described by Proposition~\ref{Prop:NonJordayTypeIrredChar} below.

\begin{prop}
\label{Prop:NonJordayTypeIrredChar}
Let $q\neq 2$ be a power of a prime $p$, and let $n$ be a composite number. Suppose a matrix $A\in\GL_n(\F_q)$ has projective order coprime to $p$, and its characteristic polynomial is irreducible. Then either $A$ has distance at most $4$ to a pivot matrix, or $q=3$ and $n=4$ and $A$ has distance at most $6$ to a pivot matrix.
\end{prop}
\begin{proof}
First let us suppose that $q\neq 3$ or $n\neq 4$.

If $q=3$ but $n\neq 4$, then we can pick a proper factor $k\geq 3$ of $n$. If $q\neq 3$, pick any proper factor $k\geq 2$ of $n$. Then by Lemma~\ref{Lem:qnPrimeEvasionQ}, either way, we can find a prime factor $p_1$ of $q^k-1$ that does not divide $q-1$.

By Lemma~\ref{Lem:CompanionOrderBasic}, since $A$ has irreducible characteristic polynomial, $\F_q[A]$ is a field. So $A$ has a root $A'$ which is a multiplicative generator of $\F_q[A]$. Since $p_1$ divides $q^k-1$, it also divides $q^n-1$. So $p_1$ divides the multiplicative order of $A'$. So $A'$ has a power $B$ with multiplicative order $p_1$. But since $p_1$ does not divide $q-1$, $B$ is not a scalar multiple of identity. Hence $B$ has projective order exactly $p_1$. 

Agains since $p_1$ divides $q^k-1$, by Lemma~\ref{Cor:CompanionSameOrderSameSize}, the characteristic polynomial of $B$ cannot be irreducible. Hence by Lemma~\ref{Lem:NonJordayTypePrimeEvasion}, $B$ has distance at most $2$ to some pivot matrix. Hence $A$ has distance at most $4$ to some pivot matrix.

Now let us consider the case when $q=3$ and $n=4$. Again by Lemma~\ref{Lem:CompanionOrderBasic}, since $A$ has irreducible characteristic polynomial, $\F_q[A]$ is a field. So any multiplicative generator of $\F_q[A]$ is a root of $A$. We can pick a multiplicative generator $A'$ of $\F_q[A]$ which has multiplicative order $80$. Since $A'$ is a root of $A$, it is also similar to the companion matrix of some polynomials. Now $(A')^20$ has multiplicative order $4$, which is not a factor of $q-1$, but a factor of $q^2-1$. So by Lemma~\ref{Lem:CompanionPower}, the minimal polynomial of $(A')^20$ is irreducible with degree $2$. In particular, $(A')^{20}$ must be similar to $\begin{bmatrix}C&\\&C\end{bmatrix}$ for some companion matrix $C$ to some irreducible polynomial. Over $\F_3$, it is easy to see that there is only one such polynomial, so we have $C=\begin{bmatrix}0&-1\\1&0\end{bmatrix}$.

Now $\begin{bmatrix}C&\\&C\end{bmatrix}$ has a $9$-th root $\begin{bmatrix}C&I\\&C\end{bmatrix}$, whose fourth power is $\begin{bmatrix}I&-C\\&I\end{bmatrix}$, which has a $4$-th root $\begin{bmatrix}1&0&0&1\\&-1&1&0\\&&-1&0\\&&&1\end{bmatrix}$, whose third power is $\begin{bmatrix}1&&&\\&-1&&\\&&-1&\\&&&1\end{bmatrix}$, a pivot matrix. So $(A')^20$ has distance at most $4$ to a pivot matrix. 

All in all, $A$ has distance at most $2$ to $(A')^{20}$, which has distance at most $4$ to a pivot matrix. So we are done.
\end{proof}

For the rest of this section, we can assume that the projective order of $A$ and $q-1$ are both powers of the same prime. However, if $q-1$ is a prime power, then by Lemma~\ref{Lem:ConseqPP}, either $q=9$, or $q$ is a Fermat prime, or $q$ is a power of $2$ with $q-1$ prime.

First we deal with the diagonalizable case.

\begin{prop}
\label{Prop:NonJordayTypeDiagOdd}
Let $q$ be Fermat prime or $9$, and let $n\geq 2$. Suppose a matrix $A\in\GL_n(\F_q)-Z(\GL_n(\F_q))$ is diagonalizable over $\F_q$, then it has distance at most $1$ to a pivot matrix.
\end{prop}
\begin{proof}
Suppose $A$ is diagonalizable with eigenvalues $\lambda_1,\dots,\lambda_n\in\F_q$, and suppose $A$ has projective order $m\neq 1$. Then $m$ divides $q-1$ and thus it is a power of $2$. So $A^{\frac{m}{2}}$ has eigenvalues $\lambda_1^{\frac{m}{2}},\dots,\lambda_n^{\frac{m}{2}}$, and all these eigenvalues have the same square. However, in the cyclic group $\F_q^*$, each element can have at most two square roots. Hence $\lambda_1^{\frac{m}{2}},\dots,\lambda_n^{\frac{m}{2}}$ can take only two possible values. Since $A^{\frac{m}{2}}$ is not a scalar multiple of identity, it must be a pivot matrix.
\end{proof}

\begin{prop}
\label{Prop:NonJordayTypeDiagEven}
Let $q$ be a power of $2$ such that $q-1$ is prime, and let $n\geq 2$. Suppose a matrix $A\in\GL_n(\F_q)-Z(\GL_n(\F_q))$ is diagonalizable over $\F_q$ but not all eigenvalues are distinct, then it has distance at most $4$ to a pivot matrix.
\end{prop}
\begin{proof}
Suppose $\lambda$ is a repeating eigenvalue for $A$. Then we have $A=X\begin{bmatrix}C_1&\\ &C_2\end{bmatrix}X^{-1}$ for some invertible $X$, where $C_1=\begin{bmatrix}\lambda&\\&\lambda\end{bmatrix}$ and $C_2$ is diagonal.

By Lemma~\ref{Lem:qnPrimeEvasionQ}, we can find a prime factor $p_1$ of $q^2-1$ coprime to $q-1$. Since $C_1$ has multiplicative order dividing $q-1$, therefore $p_1(q-1)$ divides $q^2-1$. Hence by Lemma~\ref{Lem:DoubleCompanionHasGoodRoot}, we can find a matrix $C'_1$ which is a $p_1$-th root of $C_1$, and has $p_1$ dividing its multiplicative order. Note that since $p_1$ and $q-1$ are coprime, $p_1$ must also divide the projective order of $C'_1$.

Since $q-1$ and $p_1$ are coprime, each diagonal entry of $C_2$ has a unique $p_1$-th root. So $C_2$ has a $p_1$-th root $C'_2$. Therefore $A$ has a $p_1$-th root $A'=X\begin{bmatrix}C'_1&\\ &C'_2\end{bmatrix}X^{-1}$, and since $p_1$ divides the projective order of $C'_1$, it must also divide the projective order of $A'$. Furthermore, since $A'$ is similar to a block diagonal matrix, its characteristic polynomial is reducible. So by Proposition~\ref{Prop:NonJordayTypePrimeEvasion}, $A'$ has distance at most $3$ to a pivot matrix. Hence $A$ has distance at most $4$ to a pivot matrix.
\end{proof}

Note that if $q$ be a power of $2$ such that $q-1$ is prime, and a matrix $A\in\GL_n(\F_q)-Z(\GL_n(\F_q))$ is diagonalizable over $\F_q$ with $n$ distinct eigenvalues. Then we must have $n<q$, and this is exactly the situation described by Proposition~\ref{prop:ConnObstDiag}. So we have completed all discussions on diagonalizable matrices.

We now sum up the remaining cases.

\begin{prop}
\label{Prop:NonJordayTypeBothPowerSamePrime}
Let $q$ be a power of a prime $p$, and $q-1$ is a power of a prime $p_0$, and let $n\geq 2$.  Suppose a matrix $A\in\GL_n(\F_q)-Z(\GL_n(\F_q))$ is not diagonalizable, has reducible characteristic polynomial, and its projective order is also a power of $p_0$. Then either $A$ has distance at most $4$ to a pivot matrix, or $q=3$ and $n=4$ and $A$ has distance at most $5$ to a pivot matrix.
\end{prop}
\begin{proof}
Since the projective order of $A$ is coprime to $p$, by Lemma~\ref{Lem:NonJordayTypeClassify}, the generalized Jordan canonical form of $A$ is block diagonal, with companion matrices for irreducible polynomials on its diagonal. Since $A$ is not diagonalizable, one of these diagonal blocks $C_1$ is not $1\times 1$.

Since $A$ has reducible characteristic polynomial, we must have $A=X\begin{bmatrix}C_1&\\&C_2\end{bmatrix}X^{-1}$ for some invertible $X$, where $C_2$ is block diagonal, with companion matrices for irreducible polynomials on its diagonal.

Suppose $C_1$ is $k\times k$. If $q\neq 3$, then by Lemma~\ref{Lem:qnPrimeEvasionQ}, we can find a prime factor $p_1$ of $q^k-1$ coprime to $q-1$. Since the projective order of $C_1$ and all diagonal blocks of $C_2$ are powers of $p_0$, by Lemma~\ref{Lem:PrimeCompanionHasGoodRoot}, $C_1,C_2$ have $p_1$-th root $C'_1,C'_2$. Furthremore, since $p_1,p_0$ are distinct primes dividing $q^k-1$, by Lemma~\ref{Lem:PrimeCompanionHasGoodRoot} again, we can choose $C'_1$ such that $p_1$ divides its multiplicative order. Note that since $p_1$ and $q-1$ are coprime, $p_1$ must also divide the projective order of $C'_1$.

So $A$ has a $p_1$-th root $A'=X\begin{bmatrix}C'_1&\\&C'_2\end{bmatrix}X^{-1}$ with reducible characteristic polynomial, and $p_1$ divides the projective order of $A'$. Then by Proposition~\ref{Prop:NonJordayTypePrimeEvasion}, $A'$ has distance at most $3$ to a pivot matrix. Hence $A$ has distance at most $4$ to a pivot matrix.

Now suppose $q=3$. Suppose in the diagonal blocks of the generalized Jordan canonical form of $A$, at least one block $C_1$ is $k\times k$ for some $k\geq 3$. Then again by Lemma~\ref{Lem:qnPrimeEvasionQ}, we can find a prime factor $p_1$ of $q^k-1$ coprime to $q-1$. The rest of the argument is identical to above.

Now suppose $q=3$, but all diagonal blocks of the generalized Jordan canonical form of $A$ are either $1\times 1$ or $2\times 2$. Suppose $A$ has multiplicative order $m$, which must be a power of $2$. If $A^{\frac{m}{2}}\neq -I$, then it must be a diagonalizable matrix with eigenvalues $1$ and $-1$, hence it is a pivot matrix. If $A^{\frac{m}{2}}=-I$, then for the matrix $A'=A^{\frac{m}{4}}$, its generalized Jordan canonical form must have all diagonal blocks identical to $\begin{bmatrix}0&-1\\1&0\end{bmatrix}$. 

If $n>4$, then again $A=X\begin{bmatrix}C_1&\\&C_2\end{bmatrix}X^{-1}$ where $C_1,C_2$ are both block diagonal with all diagonal blocks identical to $\begin{bmatrix}0&-1\\1&0\end{bmatrix}$, but $C_1$ only have two diagonal blocks while $C_2$ have $\frac{n-4}{2}$ diagonal blocks. Then by Lemma~\ref{Lem:DoubleCompanionHasGoodRoot}, $C_1$ has a $5$-th root $C'_1$, and $C'_1$ have projective order a multiple of $5$. We also have $C_2$ as its own $5$-th root. Hence $A$ has a $5$-th root $A'=X\begin{bmatrix}C'_1&\\&C_2\end{bmatrix}X^{-1}$ with reducible characteristic polynomial, and $5$ divides the projective order of $A'$. So again by Proposition~\ref{Prop:NonJordayTypePrimeEvasion}, $A'$ has distance at most $3$ to a pivot matrix. Hence $A$ has distance at most $4$ to a pivot matrix.

Finally, the last case requires $n=4$. Then $A'=A^{\frac{m}{4}}$ is similar to a matrix which is block diagonal, and the two diagonal blocks are both $\begin{bmatrix}0&-1\\1&0\end{bmatrix}$. We have already encountered this matrix in Proposition~\ref{Prop:NonJordayTypeIrredChar}, and it has distance at most $4$ to a pivot matrix. Hence $A$ has distance at most $5$ to a pivot matrix.
\end{proof}

\subsection{When $q=2$}

Now we turn to the case of $q=2$. Since there is no pivot matrix in this case, we would connect matrices in $\GL_n(\F_2)-Z(\GL_n(\F_2))$ to Jordan pivot matrices.

\begin{prop}
\label{Prop:JordayTypePower2}
For $n\geq 6$, if $2$ divides the multiplicative order of a matrix $A\in\GL_n(\F_2)$, then $A$ has distance at most $5$ to a Jordan pivot matrix.
\end{prop}
\begin{proof}
First, by Lemma~\ref{Lem:JordayTypePower}, $A$ has a power $A'$ whose generalized Jordan canonical form is made of diagonal blocks of the form $\begin{bmatrix}1&1\\&1\end{bmatrix}$ or $\begin{bmatrix}1\end{bmatrix}$. If $n\geq 6$ (in fact $n\geq 4$ is enough here), then either this Jordan canonical form has at least two $1\times 1$ blocks, or it has at least two $2\times 2$ blocks.

Note that $\begin{bmatrix}1&1&&\\&1&&\\&&1&1\\&&&1\end{bmatrix}$ is similar to $\begin{bmatrix}I_2&I_2\\&I_2\end{bmatrix}$. Therefore, consider the generalized Jordan canonical form of $A'$, we have $A'=X\begin{bmatrix}I_2&\\&A''\end{bmatrix}X^{-1}$ or $X\begin{bmatrix}I_2&I_2&\\&I_2&\\&&A''\end{bmatrix}X^{-1}$ for some invertible $X$, where $A''$ has multiplicative order exactly $2$. 

Set $P=\begin{bmatrix}0&1\\1&1\end{bmatrix}$ which has multiplicative order $3$. Set $B=X\begin{bmatrix}P&\\&A''\end{bmatrix}X^{-1}$ or $X\begin{bmatrix}P&P&\\&P&\\&&A''\end{bmatrix}X^{-1}$ respectively. Either way, we have $B^3=A'$, $B^6=I_n$ and $B^4=X\begin{bmatrix}P&\\&I_{n-2}\end{bmatrix}X^{-1}$ or $X\begin{bmatrix}P&&\\&P&\\&&I_{n-4}\end{bmatrix}X^{-1}$.  Let $C=X\begin{bmatrix}I_{n-2}&\\&J\end{bmatrix}X^{-1}$ where $J=\begin{bmatrix}1&1\\&1\end{bmatrix}$. Then as $n\geq 6$, $B^4$ must commute with $C$, and hence $(B^4C)^4=B^4$ and $(B^4C)^3=C$.

Then we have the following path
\[
A\xrightarrow{\text{power}} A' \xrightarrow{\text{$3$-rd root}} B \xrightarrow{\text{$4$-th power}} B^4 \xrightarrow{\text{$4$-th root}} B^4C \xrightarrow{\text{$3$-rd power}} C.
\]

So we have reached a Jordan pivot matrix in $5$ steps.
\end{proof}

\begin{lem}
\label{Lem:NonJordayTypeRepeatOne2}
For $n\geq 4$, suppose a non-identity matrix $A\in\GL_n(\F_2)$ has odd multiplicative order, and $(x-1)^2$ is a factor of the characteristic polynomial $f(x)$ of $A$. Then $A$ has distance at most $2$ to a Jordan pivot matrix.
\end{lem}
\begin{proof}
Since the projective order of $A$ is coprime to $2$, by Lemma~\ref{Lem:NonJordayTypeClassify}, the generalized Jordan canonical form of $A$ is block diagonal, with companion matrices for irreducible polynomials on its diagonal. Since $(x-1)^2$ is a factor of the characteristic polynomial $f(x)$ of $A$, therefore $A=X\begin{bmatrix}I_2&\\&C\end{bmatrix}X^{-1}$ for some invertible $X$, where $I_2$ is the $2\times 2$ identity matrix and $C$ has the same multiplicative order as $A$.

Let $k$ be the multiplicative order of $A$. Since $k$ is an odd number, therefore $A$ has a $(k+1)$-th root $A'=X\begin{bmatrix}J&\\&C\end{bmatrix}X^{-1}$ where $J=\begin{bmatrix}1&1\\&1\end{bmatrix}$, and $(A')^k$ is a Jordan pivot matrix. So $A$ has distance at most $2$ to a Jordan pivot matrix.
\end{proof}

\begin{prop}
\label{Prop:NonJordayTypePrimeOrderNonPrimeBlock2}
For $n\geq 6$, suppose the multiplicative order of a non-identity matrix $A\in\GL_n(\F_2)$ is an odd prime $p$, $p+1$ is not a power of $2$, and the characteristic polynomial of $A$ has no irreducible factor of degree $n$ or $n-1$. Then $A$ has distance at most $4$ to a Jordan pivot matrix.
\end{prop}
\begin{proof}
Let $m$ be the smallest positive integer such that $p$ divides $2^m-1$. Since $p+1$ is not a power of $2$, this means that $2^m-1$ is a composite number. In particular, by Lemma~\ref{Lem:ConseqPP}, $2^m-1$ cannot be a prime power. So we can find a prime factor $p_0\neq p$ of $2^m-1$.

Since we must have $p\neq 2$, by Lemma~\ref{Lem:NonJordayTypeClassify}, $A=X\begin{bmatrix}C_1&&\\ &\ddots&\\&&C_t\end{bmatrix}X^{-1}$ for some invertible $X$, where $C_1,\dots, C_t$ are companion matrices to various irreducible polynomials over $\F_2$. By Corollary~\ref{Cor:CompanionSameOrderSameSize}, all $C_i$ are either $1\times 1$ or $m\times m$. Since $A$ is not identity, one of these $C_i$ is $m\times m$. Say $C_1$ is $m\times m$.

Now $\F_2[C_1]$ is a field with $2^m$ elements. So $C_1$ has a $p_0$-th root $C'_1$ with multiplicative order $pp_0$. For other $C_i$, since each has multiplicative order $1$ or $p$, each has a $p_0$-th root $C'_i$ with multiplicative order $1$ or $p$. Hence $A$ has a $p_0$-th root $A'=X\begin{bmatrix}C'_1&&\\ &\ddots&\\&&C'_t\end{bmatrix}X^{-1}$, which has multiplicative order $pp_0$. 

Now, $(A')^p=X\begin{bmatrix}(C'_1)^p&\\&I_{n-m}\end{bmatrix}X^{-1}$ has multiplicative order $p_0$. Since the characteristic polynomial of $A$ has no irreducible factor of degree $n$ or $n-1$, therefore $m\leq n-2$. So $(x-1)^2$ is a factor of the characteristic polynomial $f(x)$ of $(A')^p$. Hence by Lemma~\ref{Lem:NonJordayTypeRepeatOne2}, $(A')^p$ has distance at most $2$ to a Jordan pivot matrix. Therefore, $A$ has distance at most $4$ to a Jordan pivot matrix.
\end{proof}

\begin{lem}
\label{Lem:NonJordayTypePrimeOrderMersennePrimeBlock2}
For $n\geq 6$, suppose the multiplicative order of a non-identity matrix $A\in\GL_n(\F_2)$ is an odd prime $p$, and $p+1=2^m$ for some positive integer $m$. If the minimal polynomial and the characteristic polynomial of $A$ are not the same, then either $A$ has distance at most $4$ to a Jordan pivot matrix, or $A$ is similar to $\begin{bmatrix}C&\\&C\end{bmatrix}$ or $\begin{bmatrix}C&&\\&C&\\&&1\end{bmatrix}$ for some companion matrix $C$ to an irreducible polynomial.
\end{lem}
\begin{proof}
Since we must have $p\neq 2$, by Lemma~\ref{Lem:NonJordayTypeClassify}, $A=X\begin{bmatrix}C_1&&\\ &\ddots&\\&&C_t\end{bmatrix}X^{-1}$ for some invertible $X$, where $C_1,\dots, C_t$ are companion matrices to various irreducible polynomials over $\F_2$. By Corollary~\ref{Cor:CompanionSameOrderSameSize}, all $C_i$ are either $1\times 1$ or $m\times m$. Since the minimal polynomial and the characteristic polynomial of $A$ are not the same, two of these $C_i$ are identical. Say $C_1=C_2$.

If $C_1,C_2$ are both $1\times 1$, then $(x-1)^2$ divides the characteristic polynomial of $A$. Hence $A$ has distance at most $2$ to a Jordan pivot matrix.

Now suppose $C_1,C_2$ are both $m\times m$. Note that $2^{2m}-1=(2^m-1)(2^m+1)$, and $2^m-1,2^m+1$ must be coprime. Hence we can find a prime factor $p_0\neq p$ of $2^m+1$. By Lemma~\ref{Lem:DoubleCompanionHasGoodRoot}, we can find a matrix $C'$ such that $(C')^{p_0}=\begin{bmatrix}C_1&\\&C_2\end{bmatrix}$ and $p_0$ divides the multiplicative order of $C'$.

For all $i>2$, since $p_0\neq p$, $C_i$ has a $p_0$-th root $C'_i$ by Lemma~\ref{Lem:PrimeCompanionHasGoodRoot}. So all in all, $A$ has a $p_0$-th root $A'=X\begin{bmatrix}C'&&&\\ &C'_3&&\\&&\ddots&\\&&&C'_t\end{bmatrix}X^{-1}$, and $(A')^p=X\begin{bmatrix}(C')^p&\\ &I_{n-2m}\end{bmatrix}X^{-1}$. Since $p_0\neq p$ divides the multiplicative order of $C'$, we see that $(A')^p$ is not an identity matrix.

Suppose that $A$ is not similar to $\begin{bmatrix}C&\\&C\end{bmatrix}$ or $\begin{bmatrix}C&&\\&C&\\&&1\end{bmatrix}$ for some companion matrix $C$ to an irreducible polynomial. Then the matrix $\begin{bmatrix}C_3&&\\ &\ddots&\\&&C_t\end{bmatrix}$ is at least $2\times 2$, and hence $n-2m\geq 2$. Therefore $(A')^p$ satisfies the requirement of Lemma~\ref{Lem:NonJordayTypeRepeatOne2}. Hence $(A')^p$ has distance at most $2$ to a Jordan pivot matrix, and $A$ has distance at most $4$ to a Jordan pivot matrix.
\end{proof}

\begin{prop}
\label{Prop:NonJordayTypePrimeOrderMersennePrimeBlock2}
For $n\geq 6$, suppose the multiplicative order of a non-identity matrix $A\in\GL_n(\F_2)$ is an odd prime $p$, and $p+1=2^m$ for some positive integer $m$. If the minimal polynomial and the characteristic polynomial of $A$ are not the same, then $A$ has distance at most $6$ to a Jordan pivot matrix.
\end{prop}
\begin{proof}
By Lemma~\ref{Lem:NonJordayTypePrimeOrderMersennePrimeBlock2}, we only need to consider the case when $A$ is similar to $\begin{bmatrix}C&\\&C\end{bmatrix}$ or $\begin{bmatrix}C&&\\&C&\\&&1\end{bmatrix}$ for some $m\times m$ companion matrix $C$ to an irreducible polynomial. Since $n\geq 6$, this implies that $m\geq 3$. So $3$ divides $2^{2m}-1$ but $3\neq p$.

By Lemma~\ref{Lem:DoubleCompanionHasGoodRoot}, we can find a matrix $C'$ such that $(C')^{3}=\begin{bmatrix}C&\\&C\end{bmatrix}$ and the multiplicative order of $C'$ is exactly $3p$. Therefore, if $A$ is similar to $\begin{bmatrix}C&\\&C\end{bmatrix}$ or $\begin{bmatrix}C&&\\&C&\\&&1\end{bmatrix}$, then $A$ has a third root $A'$ similar to $C'$ or $\begin{bmatrix}C'&\\ &1\end{bmatrix}$. In particular, $A'$ has multiplicative order $3p$, and $(A')^p$ has multiplicative order $3$. Its minimal polynomial has degree at most $3$, so it cannot equal to its characteristic polynomial. 

Furthermore, since $3=2^2-1$, by Corollary~\ref{Cor:CompanionSameOrderSameSize}, the generalized Jordan canonical form of $(A')^p$ must be block diagonal by $1\times 1$ or $2\times 2$ blocks. Since $n\geq 6$, it cannot be similar to $\begin{bmatrix}C''&\\&C''\end{bmatrix}$ or $\begin{bmatrix}C''&&\\&C''&\\&&1\end{bmatrix}$ for some $2\times 2$ companion matrix $C''$ to an irreducible polynomial. So by Lemma~\ref{Lem:NonJordayTypePrimeOrderMersennePrimeBlock2}, $(A')^p$ has distance at most $4$ to a Jordan pivot matrix. Hence $A$ has distance at most $6$ to a Jordan pivot matrix.
\end{proof}

Note that in the case above, if the minimal polynomial and the characteristic polynomial of $A$ are the same, then this is exactly the case described by Proposition~\ref{prop:ConnObstDiag2}.

\begin{prop}
\label{Prop:NonJordayTypeNonPrimeOrder2}
For $n\geq 6$, suppose the multiplicative order of a non-identity matrix $A\in\GL_n(\F_2)$ is an odd composite number, and the characteristic polynomial of $A$ has no irreducible factor of degree $n$ or $n-1$. Then $A$ has distance at most $5$ to a Jordan pivot matrix.
\end{prop}
\begin{proof}
Suppose the multiplicative order of $A$ has a prime factor $p$ such that $p+1$ is not a power of $2$, then $A$ has a power $A'$ of multiplicative order $p$. By Lemma~\ref{Lem:IrreducibleRootIrreducible}, the characteristic polynomial of $A'$ has no irreducible factor of degree $n$ or $n-1$. Therefore, by Proposition~\ref{Prop:NonJordayTypePrimeOrderNonPrimeBlock2}, $A'$ has distance at most $4$ to a Jordan pivot matrix. So $A$ has distance at most $5$ to a Jordan pivot matrix.

Now suppose all prime factors of the multiplicative order of $A$ are one less than some powers of two. Say the multiplicative order of $A$ has a prime factor $p=2^m-1$ for some positive integer $m$. Then $A$ has a power $A'$ of multiplicative order $p$. 

By Lemma~\ref{Lem:NonJordayTypeClassify}, $A=X\begin{bmatrix}C_1&&\\ &\ddots&\\&&C_t\end{bmatrix}X^{-1}$ for some invertible $X$, where $C_1,\dots, C_t$ are companion matrices to various irreducible polynomials over $\F_2$. Let $k$ be the multiplicative order of $A$, then $A'=A^{\frac{k}{p}}=X\begin{bmatrix}C_1^{\frac{k}{p}}&&\\ &\ddots&\\&&C_t^{\frac{k}{p}}\end{bmatrix}X^{-1}$. 

Note that each $C_i^{\frac{k}{p}}$ has multiplicative order $1$ or $p=2^m-1$. If all $C_i^{\frac{k}{p}}$ have irreducible characteristic polynomials, then by Corollary~\ref{Cor:CompanionSameOrderSameSize}, they all have size $1\times 1$ or $m\times m$. But then $C_i$ must have multiplicative order dividing $1$ or $p=2^m-1$. So $A$ has prime multiplicative order, contradiction. So for some $C_i$, $C_i^{\frac{k}{p}}$ no longer has irreducible characteristic polynomial. In particular, $A'$ must have distinct minimal polynomial and characteristic polynomial.

Suppose for contradiction that $A'$ is similar to $\begin{bmatrix}C&\\&C\end{bmatrix}$ or $\begin{bmatrix}C&&\\&C&\\&&1\end{bmatrix}$ for some companion matrix $C$ of some irreducible polynomial. Since $C$ must have multiplicative order $p$, it must be $m\times m$. Since the characteristic polynomial of $A$ has no irreducible factor of degree $n$ or $n-1$, compare with the generalized Jordan canonical forms of $A$ and of $A'$, we can only have $A=\begin{bmatrix}C_1&\\&C_2\end{bmatrix}$ or $\begin{bmatrix}C_1&&\\&C_2&\\&&1\end{bmatrix}$, where $C_1,C_2$ are both $m\times m$. But since $p=2^{m}-1$, the multiplicative orders of $C_1,C_2$ must divide $p$. So again $A$ has prime multiplicative order, contradiction.  So $A'$ is not similar to $\begin{bmatrix}C&\\&C\end{bmatrix}$ or $\begin{bmatrix}C&&\\&C&\\&&1\end{bmatrix}$ for some companion matrix $C$ of some irreducible polynomial. 

In conclusion, Lemma~\ref{Lem:NonJordayTypePrimeOrderMersennePrimeBlock2} must apply to $A'$. So $A'$ has distance at most $4$ to a Jordan pivot matrix. Thus $A$ has distance at most $5$ to a Jordan pivot matrix.
\end{proof}

\begin{lem}
\label{Lem:NonJordayTypeIrredFactorN2}
For a compositite number $n\geq 6$, suppose the multiplicative order of a non-identity matrix $A\in\GL_n(\F_2)$ is an odd number. If the characteristic polynomial of $A$ is irreducible, then $A$ has distance at most $6$ to a Jordan pivot matrix.
\end{lem}
\begin{proof}
Since the characteristic polynomial of $A$ is irreducible, $\F_q[A]$ is a field with $q^n$ elements. Let $A'$ be a multiplicative generator of this field (and hence a root of $A$).

Since $n\geq 6$ is a composite number, let $m\geq 2$ be its smallest prime factor, then $n\geq 3m\geq 2m+2$. Pick any prime factor $p$ dividing $2^m-1$. Then $p$ also divides $2^n-1$, and therefore $A'$ has a power $A''$ whose multiplicative order is $p$. By Lemma~\ref{Lem:NonJordayTypeClassify}, $A''=X\begin{bmatrix}C_1&&\\ &\ddots&\\&&C_t\end{bmatrix}X^{-1}$ for some invertible $X$, where $C_1,\dots, C_t$ are companion matrices to various irreducible polynomials over $\F_q$. Furthermore, since each $C_i$ has multiplicative order dividing $p$, by Corollary~\ref{Cor:CompanionSameOrderSameSize}, therefore it is either $1\times 1$ or $m\times m$. 

Since $n\geq 3m$, we must have $t\geq 3$, so the characteristic polynomial of $A''$ cannot have irreducible factor of degree $n$ or $n-1$. If $p\neq 2^m-1$, then by Proposition~\ref{Prop:NonJordayTypePrimeOrderNonPrimeBlock2}, $A''$ has distance at most $4$ to a Jordan pivot matrix. Therefore $A$ has distance at most $6$ to a Jordan pivot matrix.

If $p=2^m-1$, note that $n\geq 2m+2$, so $A''$ cannot be similar to $\begin{bmatrix}C&\\&C\end{bmatrix}$ or $\begin{bmatrix}C&&\\&C&\\&&1\end{bmatrix}$ for some companion matrix $C$ of some irreducible polynomial. Furthermore, $A''$ is a power of $A'$, so by Lemma~\ref{Lem:CompanionPower}, its minimal polynomial must be irreducible. So its minimal polynomial and characteristic polynomial are different. Hence by Lemma~\ref{Lem:NonJordayTypePrimeOrderMersennePrimeBlock2}, $A''$ has distance at most $4$ to a Jordan pivot matrix. Therefore $A$ has distance at most $6$ to a Jordan pivot matrix.
\end{proof}

\begin{prop}
\label{Prop:NonJordayTypeIrredFactorN2}
For $n\geq 6$, suppose the multiplicative order of a non-identity matrix $A\in\GL_n(\F_2)$ is an odd number. If the characteristic polynomial of $A$ has an irreducible of degree $n$ or $n-1$, and this degree is a composite number, then $A$ has distance at most $6$ to a Jordan pivot matrix.
\end{prop}
\begin{proof}
By Lemma~\ref{Lem:NonJordayTypeIrredFactorN2}, we only need to prove the case when the characteristic polynomial of $A$ has an irreducible of degree $n-1$, and $n-1$ is a composite number. Then the generalized Jordan canonical form of $A$ must be $\begin{bmatrix}C&\\&1\end{bmatrix}$ where $C$ is a companion matrix to an irreducible polynomial of degree $n-1$.

Since $n-1$ is a composite number, we cannot have $n=6$. So $n\geq 7$. Then $C$ has irreducible characteristic polynomial, and its dimension is $n-1\geq 6$. Therefore by Lemma~\ref{Lem:NonJordayTypeIrredFactorN2}, $C$ has distance at most $6$ to a Jordan pivot matrix. Since $A$ is similar to $\begin{bmatrix}C&\\&1\end{bmatrix}$ and $1$ is its own power and root, $A$ has distance at most $6$ to a Jordan pivot matrix.
\end{proof}

\bibliographystyle{alpha}
\bibliography{references}

\end{document}